\documentclass[10pt,a4paper,reqno]{amsart}
\usepackage{amsmath}
\usepackage{amsfonts}
\usepackage{amssymb}
\usepackage{amsthm}
\usepackage{mathrsfs}
\usepackage[shortlabels]{enumitem}
\usepackage{graphicx}
\usepackage{color}
\usepackage{dsfont}
\usepackage[latin1]{inputenc}
\usepackage{MnSymbol}
\usepackage{enumitem}
\usepackage{extpfeil}
\usepackage{mathtools}
\usepackage{url}
\usepackage[margin=1.25in]{geometry}
\usepackage{fancyhdr}
\usepackage[all,cmtip]{xy}
\usepackage{hyperref}
\usepackage{ifthen}
\usepackage{cleveref}
\usepackage{thm-restate}
\usepackage{tikz}
\usetikzlibrary{decorations.pathmorphing}
\usetikzlibrary{decorations.markings,arrows,arrows.meta,cd,math}
\tikzset{->-/.style={decoration={
  markings,
  mark=at position #1 with {\arrow{Stealth[length=7pt,width=4.67pt]}}},postaction={decorate}}}
\tikzset{mapping/.style={decoration={
  markings,
  mark=at position #1 with {\arrow{Classical TikZ Rightarrow[length=3pt]}}},postaction={decorate}}}
\tikzset{snake it/.style={decorate, decoration=snake}}
\usetikzlibrary{decorations.pathreplacing}

\numberwithin{equation}{section}

\newcommand{\R}{{\mathds R}}

\newcommand{\Z}{{\mathds Z}}

\newcommand{\Area}{\operatorname{Area}}
\newcommand{\Radius}{\operatorname{Radius}}
\newcommand{\lk}{\operatorname{lk}}

\def\ker{{\rm{ker}}}

\def\tX{\widetilde{X}}

\theoremstyle{plain}
\newtheorem{theorem}{Theorem}[section]

\newtheorem{corollary}[theorem]{Corollary}

\newtheorem{lemma}[theorem]{Lemma}
\newtheorem{proposition}[theorem]{Proposition}

\newtheorem*{question*}{Question}

\theoremstyle{definition}
\newtheorem{remark}[theorem]{Remark}
\newtheorem*{acknowledgements*}{Acknowledgements}

\newtheorem{definition}[theorem]{Definition}
\newtheorem*{notation*}{Notation}
\newtheorem*{convention*}{Convention}

\title{Explicit polynomial bounds on Dehn functions of subgroups of hyperbolic groups}

\author{Robert Kropholler}
\address{Robert Kropholler, Mathematics Institute, Zeeman Building, University of Warwick, Coventry CV4 7AL, United Kingdom}
\email{robertkropholler@gmail.com}
\author{Claudio Llosa Isenrich}
\address{Claudio Llosa Isenrich, Faculty of Mathematics, Karlsruhe Institute of Technology, Englerstr. 2, 76131 Karlsruhe, Germany}
\email{claudio.llosa@kit.edu}
\author{Ignat Soroko}
\address{Ignat Soroko, Department of Mathematics, University of North Texas, Denton, TX 76203-5017, USA}
\email{ignat.soroko@unt.edu, ignat.soroko@gmail.com}

\thanks{}

\begin{document}

\begin{abstract}
    In 1999 Brady constructed the first example of a non-hyperbolic finitely presented subgroup of a hyperbolic group by fibring a non-positively curved cube complex over the circle. We show that his example has Dehn function bounded above by $n^{96}$. This provides the first explicit polynomial upper bound on the Dehn function of a finitely presented non-hyperbolic subgroup of a hyperbolic group. We also determine the precise hyperbolicity constant for the $1$-skeleton of the universal cover of the cube complex in Brady's construction with respect to the $4$-point condition for hyperbolicity.
\end{abstract}

\maketitle

\section{Introduction}

A classical quasi-isometry invariant of finitely presented groups are Dehn functions. Among groups studied in geometric group theory, hyperbolic groups play a distinguished role. They can be characterised among all finitely presented groups by the linearity of their Dehn function~\cite{Gro-87}. Hyperbolic groups share many strong properties, and a natural problem that has recently attracted much attention is to understand the geometry of their subgroups. In this work, we contribute towards this problem, by deriving the first explicit polynomial upper bound on the Dehn function of a finitely presented non-hyperbolic subgroup of a hyperbolic group.
\begin{theorem}\label{thm:Main-Dehn}
    There is a hyperbolic group $G$ with a finitely presented subgroup $H\leq G$ which has Dehn function $\delta_H$ satisfying $n^2\preccurlyeq \delta_H(n)\preccurlyeq n^{96}$. 
\end{theorem}

The groups $G$ and $H$ in \Cref{thm:Main-Dehn} were constructed by Brady~\cite{Bra-99}. They provided the first example of a hyperbolic group $G$ with a non-hyperbolic finitely presented subgroup $H$. The group $G$ is a fundamental group of a non-positively curved cube complex $X$. A key step in our proof of \Cref{thm:Main-Dehn}, which seems of independent interest, is the calculation of the hyperbolicity constant of the $1$-skeleton of the universal cover of $X$.
\begin{theorem}\label{thm:Main-hyperbolicity-constant}
    Let $\tX^{(1)}$ be the $1$-skeleton of the universal cover of the cube complex $X$. Its optimal hyperbolicity constant for the $4$-point condition is $4$.
\end{theorem}

We will now provide further context and motivation for our results.

\subsection{Hyperbolic groups and their non-hyperbolic subgroups}

Hyperbolic groups were introduced by Gromov~\cite{Gro-87} and their study has formed a very active area of geometric group theory ever since. We refer to \Cref{sec:hyperbolicity} for their precise definition. Generalising fundamental groups of closed manifolds with negative sectional curvature, hyperbolic groups form a large class of groups which share many strong properties. For instance, they have solvable word problem, solvable conjugacy problem, satisfy a strong version of the Tits Alternative, do not contain $\mathbb{Z}^2$, and are of finiteness type $F_{\infty}$ \cite{BriHae-99}. The strong properties shared by all hyperbolic groups raise the natural problem if they are inherited by their subgroups.

Recently, many examples of non-hyperbolic finitely presented subgroups of hyperbolic groups have been constructed, and a key tool in their construction have been finiteness properties. We call a group $G$ of {\em finiteness type $F_n$} if it can be realised as a fundamental group of an aspherical CW-complex $X$ with finitely many cells of dimension $\leq n$. We say that $G$ is of {\em type $F_{\infty}$} if it is $F_n$ for all $n\geq 0$ and of {\em type $F$} if $X$ can be chosen to be finite. The finiteness properties $F_n$ generalise finite generation, which is equivalent to $F_1$, and finite presentability, which is equivalent to $F_2$. 

Since hyperbolic groups are of type $F_{\infty}$, any subgroup that is not of type $F_n$ for some $n$ is non-hyperbolic. The first finitely generated non-hyperbolic subgroups of hyperbolic groups were constructed by Rips~\cite{Rip-82} and their non-hyperbolicity follows from this criterion, as they are not finitely presented. In 1999 Brady~\cite{Bra-99} produced the first example of a non-hyperbolic finitely presented subgroup of a hyperbolic group, which is the group covered by \Cref{thm:Main-Dehn,thm:Main-hyperbolicity-constant}. His group is of type $F_2$, but not $F_3$. More examples of this kind were obtained by Lodha~\cite{Lod-18} and Kropholler~\cite{Kro-21}. Recently, many further examples of finitely presented non-hyperbolic subgroups of hyperbolic groups were constructed by Llosa Isenrich, Martelli and Py, and by Llosa Isenrich and Py, showing that they can be of type $F_n$ and not $F_{n+1}$ for every $n\geq 1$ \cite{LIMP-21,LloPy-22}. Moreover, the first example of a non-hyperbolic subgroup of a hyperbolic group of type $F$ was constructed by Italiano, Martelli and Migliorini~\cite{IMM-22,IMM-21}. All examples mentioned so far arise as kernels of morphisms from a hyperbolic group onto $\mathbb{Z}$. Very recently, Kropholler and Llosa Isenrich~\cite{KroLlo-23} and Llosa Isenrich and Py~\cite{LloPy-23} constructed examples arising as kernels of homomorphisms onto $\mathbb{Z}^k$ for $k\geq 2$.

\subsection{Dehn functions of subgroups of hyperbolic groups}

The new pool of examples of non-hyperbolic finitely presented subgroups of hyperbolic groups naturally raises the question whether it allows us to obtain a deeper understanding of their geometry. Since hyperbolic groups are characterised by the linearity of their Dehn function, all of the above examples have superlinear Dehn function. This raises the problem if one can attain a more precise understanding of the Dehn functions of subgroups of hyperbolic groups. Before explaining this further, we recall the definition of the Dehn function of a finitely presented group. 

Let $F_S$ be the free group on the set of free generators $S$, and let $\ell$ be the word length on $F_S$ with respect to $S$. For a finitely presented group $G$ with presentation $\mathcal{P}=\left\langle S\mid R\right\rangle$ we say that a word $w$ in $S^{\pm 1}$ is {\em null-homotopic} if it represents the trivial element of $G$.  We define its area as
\[
    \Area_{\mathcal{P}}(w):= \min\Bigl\{ k\, \Big\mid\, w=_{F_{S}} \prod_{i=1}^k u_i \cdot r_i^{\epsilon_i} \cdot u_i^{-1},\, u_i\in F_S,\, r_i\in R,\, \epsilon_i=\pm 1\Big\}.
\]
The {\em Dehn function} $\delta_{\mathcal{P}}$ of $G$ with respect to the presentation $\mathcal{P}$ is then defined as
\[
    \delta_{\mathcal{P}}(n):=\max\bigl\{\Area_{\mathcal{P}}(w)\,\big\mid\, w=_{G} 1,~ \ell(w)\leq n \bigr\}.
\]
It measures the maximal number of conjugates of relations from $R$ that are required to detect if a word of length at most $n$ in $S^{\pm 1}$ represents the trivial element in $G$. The Dehn function is invariant under quasi-isometries up to a certain asymptotic equivalence of functions (see \Cref{sec:BG-Dehn}). Thus, its equivalence class is independent of the choice of a finite presentation, and we usually speak of the Dehn function $\delta_G$ of the group $G$, whenever this does not lead to any confusion.

While the problem of understanding Dehn functions of finitely presented subgroups of hyperbolic groups was first raised in 1999 by Brady in the context of his example~\cite{Bra-99}, our understanding is still very limited. It has been shown by Gersten and Short~\cite{GerSho-02} that every subgroup of a hyperbolic group $G$ arising as kernel $H$ of a homomorphism onto $\mathbb{Z}$ has Dehn function bounded above by some polynomial function $p$. This has recently been generalised to kernels of homomorphisms onto arbitrary finitely generated abelian groups by Llosa Isenrich~\cite{Llo-24}. As a consequence, all known examples of finitely presented subgroups of hyperbolic groups have a polynomially bounded Dehn function. However, to the best of our knowledge, there is not a single example for which we have an explicit upper bound on the degree of this polynomial. \Cref{thm:Main-Dehn} gives the first such bound.

\subsection{Idea of proof of \texorpdfstring{\Cref{thm:Main-Dehn}}{}}

 A careful analysis of Gersten--Short's proof in~\cite{GerSho-02} reveals that the degree of the polynomial upper bound $p$ depends on several constants related to the hyperbolicity constant $\delta$ in the $\delta$-thinness characterisation of hyperbolicity and the growth of relations of $H$ under the $\Z$-action by conjugation on them. The explicit nature of these constants suggests that for a concrete example one should be able to determine an explicit upper bound on the degree of $p$. However, this requires determining the constants associated with a suitable explicit presentation of $H$. In practice, this seems rather challenging. 
 
 To circumvent this challenge and prove \Cref{thm:Main-Dehn}, we pursue a different approach, which relies on adapting a more geometric point of view on Gersten--Short's work~\cite{GerSho-02}. Instead of filling conjugates of algebraic relations, we geometrically push fillings in ${\rm CAT}(0)$ cube complexes. This is a technique that was first used by Abrams, Brady, Dani, Duchin and Young~\cite{ABDDY-13} to give a geometric proof of Dison's~\cite{Dis-08} quartic upper bound on Dehn functions of Bestvina--Brady groups. Our geometric approach produces an upper bound on the Dehn function of Brady's group in terms of two constants associated with the geometry of the corresponding ${\rm CAT}(0)$ cube complex. The first is related to the filling area of loops in the descending and ascending links of a suitable height function on the cube complex. The second is the hyperbolicity constant for the $1$-skeleton of the cube complex. 

While the constant coming from the ascending and descending links follows directly from their description, determining the hyperbolicity constant requires more work. A key step in our proof of \Cref{thm:Main-Dehn} is thus the proof of \Cref{thm:Main-hyperbolicity-constant}, which provides the precise value for the hyperbolicity constant of the $1$-skeleton of the ${\rm CAT}(0)$ cube complex constructed by Brady in~\cite{Bra-99} with respect to the $4$-point condition for hyperbolicity (see \Cref{sec:hyperbolicity} for its definition). The result, as well as its proof, seems of independent interest and may have consequences beyond its application to Dehn functions. 

We finish this section by mentioning that our techniques also apply to other examples of finitely presented non-hyperbolic subgroups of hyperbolic groups constructed from ${\rm CAT}(0)$ cube complexes, such as the ones of Lodha~\cite{Lod-18} and Kropholler~\cite{Kro-21} (see \Cref{rmk:Lodha-Kropholler-examples}). 

\subsection{Structure}
In \Cref{sec:Background} we introduce all relevant notions and examples. We start by giving further details on Dehn functions in \Cref{sec:BG-Dehn} and on hyperbolic spaces and groups in \Cref{sec:hyperbolicity}. We then introduce ${\rm CAT}(0)$ cube complexes and Morse theory in \Cref{sec:Cat-0-cube-complexes}, before discussing Brady's example in \Cref{sec:Bradys-group}. In \Cref{sec:geometric-Gersten-Short} we explain our geometric version of Gersten--Short's work~\cite{GerSho-02}. This provides a proof of \Cref{thm:Main-Dehn} modulo \Cref{thm:Main-hyperbolicity-constant}. In \Cref{sec:upper-bound-on-hyperbolicity-constant} we then determine the hyperbolicity constant for the cube complex in Brady's example, proving \Cref{thm:Main-hyperbolicity-constant}.

\subsection{Acknowledgements}
We are very grateful to Noel Brady for many insightful conversations during our work on this article that, in particular, provided inspiration for our proof of the upper bound on the hyperbolicity constant.

\section{Background}\label{sec:Background}
In this section we will introduce all preliminary notions and examples that we will require subsequently, including Dehn functions of CW-complexes, hyperbolic spaces and groups, Morse theory on ${\rm CAT}(0)$ cube complexes and Brady's example of a non-hyperbolic finitely presented subgroup of a hyperbolic group.

\subsection{Dehn functions}\label{sec:BG-Dehn}
In the introduction we introduced the Dehn function $\delta_{\mathcal{P}}$ of a finitely presented group $G$ with respect to a finite presentation $\mathcal{P}=\left\langle S\mid R \right \rangle$. The Dehn function is independent of the choice of finite presentation for $G$ up to the equivalence $\asymp$ of non-decreasing functions $f,~g : \mathbb{N}\to \mathbb{N}$ defined by $f\preccurlyeq g$ if there is a $C>0$ such that for all $n\in \mathbb{N}$ we have $f(n)\leq Cg(Cn+C)+Cn+C$. 

There are also more geometric points of view on Dehn functions in terms of optimal isoperimetric functions. One such approach is via van Kampen diagrams. We will briefly summarise the most important facts and results here and refer the reader to~\cite{Alo-90, Bri-02, BRS-07} and the references therein for further details and definitions. 

From now on we will always equip the $1$-skeleton $X^{(1)}$ of a CW-complex $X$ with the length metric obtained by identifying every edge with the unit interval. By restriction, this also equips $X^{(0)}$ with a metric. We start by recalling that a van Kampen diagram $D$ for an edge loop $\gamma : S^1 \to X^{(1)}$ in a CW-complex $X$ is a planar compact $2$-complex $D$ together with a cellular map $f: D \to X$ such that $f|_{\partial D}=\gamma$. The area $\Area(D)$ of $D$ is the number of $2$-cells of $D$. The radius $\Radius(D)$ is the maximal distance of a vertex of $D$ from the boundary $\partial D$.

One can define the Dehn function $\delta_{\mathcal{P}}(n)$ geometrically using van Kampen diagrams. That is, $\delta_{\mathcal{P}}(n)$ is the minimal area of a van Kampen diagram for any edge loop of length $\leq n$. From this geometric point of view, we can extend the notion of Dehn function to any simply connected CW-complex $X$ as follows. Given a null-homotopic loop $\gamma$ in $X^{(1)}$ we can define its area as
\[
    \Area(\gamma):= \min \left\{\Area(D) \mid D \mbox{ is a van Kampen diagram for } \gamma \mbox{ in } X^{(2)}\right\} \in \mathbb{N}.
\]
The {\em Dehn function} $\delta_X$ of $X$ is then defined as 
\[
    \delta_X(n):= \max \left\{\Area(\gamma)\mid \gamma \mbox{ is a loop of length} \leq n \mbox{ in } X^{(1)}\right\}\in \mathbb{N}\cup \left\{\infty\right\}.
\]

\begin{definition}
    A simply connected CW-complex $X$ is a \emph{Dehn complex} if $\delta_X(n)$ is finite for all $n$ and the supremum over the possible number of edges in the boundary of a $2$-cell in $X$ is finite.
\end{definition}

An important class of Dehn complexes are simply connected locally finite cell complexes equipped with a group action which is cocompact on $X^{(2)}$. In particular, every simply connected cell complex equipped with a properly discontinuous cocompact group action is a Dehn complex.

We say that two Dehn complexes $X$ and $Y$ are \emph{quasi-isometric} if the metric spaces $X^{(0)}$ and $Y^{(0)}$ are quasi-isometric. This generalises quasi-isometry of groups, as the universal cover of a presentation complex for $G$ with respect to a presentation $\mathcal{P}=\left\langle S \mid R \right\rangle$ is a Dehn complex with $X^{(0)}$ isometric to $G$ equipped with the word metric induced by $S$.

The main result of~\cite{Alo-90} is.
\begin{theorem}[{Alonso~\cite[Theorem 2.2]{Alo-90}}]\label{thm:Alonso}
    Let $X$ and $Y$ be Dehn complexes. If $X$ and $Y$ are quasi-isometric, then $\delta_X\asymp \delta_Y$.
\end{theorem}
A straightforward consequence of Theorem \ref{thm:Alonso} and the above discussion is the following result, which we will use in the proof of Theorem \ref{thm:Main-Dehn}.
\begin{corollary}\label{cor:Alonso}
    Let $G$ be a finitely presented group, and let $\mathcal{P}$ be a finite presentation. If $G$ acts properly discontinuously and cocompactly on a simply connected cell complex $X$, then $\delta_X\asymp \delta_{\mathcal{P}}$.
\end{corollary}

\subsection{Hyperbolic spaces and groups}\label{sec:hyperbolicity}

The hyperbolicity of a metric space can be defined in multiple equivalent ways. When transitioning between them, the associated hyperbolicity constant of the space can change. In this work we will require several of these equivalent definitions, and it will be important that we have explicit upper bounds on the associated hyperbolicity constants. We thus introduce these different definitions of hyperbolicity and discuss the relation between their associated constants. We refer to Bridson--Haefliger~\cite{BriHae-99} and de la Harpe--Ghys~\cite{dlHGhys-90} for a detailed discussion of hyperbolicity and proofs of the following results.

Let $(X,d)$ be a geodesic metric space. For $x\in X$ we define the \emph{Gromov product} of $y,z\in X$ with respect to $x$ as 
\[
    (y\cdot z)_x:= \frac{1}{2} \left\{d(y,x)+d(z,x)-d(y,z)\right\}.
\]

\begin{definition}\label{def:delta-hyp} 
Let $\delta\ge0$.
    A metric space $X$ is said to be \emph{$\delta$-hyperbolic} if for all $w, x, y, z\in X$
    \begin{equation}\label{eqn:Gromov-hyperbolicity}
        (x\cdot y)_w\geq \min\left\{(x\cdot z)_w, (y\cdot z)_w\right\} - \delta.
    \end{equation}
\end{definition}

\begin{remark}
    One can show (see~\cite[1.1.B]{Gro-87}) that if there is a $\delta\geq 0$ such that $X$ satisfies inequality \eqref{eqn:Gromov-hyperbolicity} for arbitrary $x,y,z\in X$ and for some fixed $w_0\in X$, then $X$ is $2\delta$-hyperbolic. Thus, it suffices to check inequality \eqref{eqn:Gromov-hyperbolicity} for a fixed point $w_0\in X$.
\end{remark}

We say that $X$ \emph{satisfies the $4$-point $\delta$-condition} if there is a $\delta\geq 0$ such that for all $w, x, y, z\in X$ the following inequality holds:
\begin{equation} \label{eqn:4-point-condition}
    d(x,y)+d(z,w) \leq \max\left\{d(x,z)+d(y,w), d(y,z)+d(x,w)\right\} + 2\delta.
\end{equation}
To explain this condition, we think of the points $x, y, z, w$ as the four vertices of a tetrahedron and we call $0\leq S\leq M\leq L$ the sums of the lengths of the three pairs of opposite edges. Then inequality \eqref{eqn:4-point-condition} is equivalent to $L-M\leq 2\delta$. 

A straightforward computation shows the equivalence of \eqref{eqn:Gromov-hyperbolicity} and \eqref{eqn:4-point-condition}, eliminating the asymmetry of the roles of points $x,y,z,w$ in Definition~\ref{def:delta-hyp}. The following lemma records this.

\begin{lemma}
    A metric space $X$ is $\delta$-hyperbolic if and only if $X$ satisfies the $4$-point $\delta$-condition.
\end{lemma}

Alternative definitions of hyperbolicity are in terms of slim triangles,  thin triangles, and insizes. They provide a more intuitive view on the nature of hyperbolic spaces. Here we will require the definitions in terms of thin triangles and insizes that we shall now give. For the definition in terms of slim triangles and the proof of its equivalence to the definitions given here we refer to~\cite[III.H.1]{BriHae-99}.

A \emph{tripod} is a metric tree with three vertices of valence one and one vertex of valence three. We extend this definition to allow some edges to have length zero, that is, we allow all metric trees with at most three edges and at most one vertex of valence greater than one. We call a tripod that does not have a vertex of degree three a degenerate tripod.
For a geodesic triangle $\Delta$ in $X$ we can define a (possibly degenerate) tripod $T_{\Delta}$ with the same side lengths and a map $\chi_{\Delta}: \Delta\to T_{\Delta}$ which is an isometry on the sides. We say that $\Delta$ is \emph{$\delta$-thin} if for each $t\in T_{\Delta}$ the set $\chi_{\Delta}^{-1}(t)$ has diameter $\leq \delta$. We say that $X$ \emph{has $\delta$-thin triangles} if each geodesic triangle in $X$ is $\delta$-thin. 

With the same notation as above, the centre $o_{\Delta}\in T_{\Delta}$ of the tripod $T_{\Delta}$ is the unique central vertex of $T_{\Delta}$, if $T_{\Delta}$ is a non-degenerate tripod, and the unique vertex lying on the geodesic between the other two vertices, if $T_{\Delta}$ is degenerate. The \emph{insize} of $\Delta$ is the diameter of $\chi_{\Delta}^{-1}(o_{\Delta})$; we denote it by ${\rm insize}(\Delta)$.

We will later make use of the following quantitative equivalences between $\delta$-hyperbolicity and the $\delta$-thin triangle condition.

\begin{lemma}
    If $X$ has $\delta$-thin triangles, then $X$ is $\delta$-hyperbolic.
\end{lemma}
\begin{proof}
    One first shows that $X$ satisfies ${\rm insize}(\Delta)\leq \delta$ for every geodesic triangle if and only if $X$ has $\delta$-thin triangles. The ``if'' implication in this assertion is obvious and the ``only if'' direction follows from the proof of~\cite[Proposition III.H.1.17]{BriHae-99}.
    The assertion then follows from the proof of~\cite[Proposition III.H.1.22]{BriHae-99}.
\end{proof}

\begin{lemma}\label{lem:thin-hyperbolicity}
    If $X$ is $\delta$-hyperbolic, then $X$ has $4\delta$-thin triangles.
\end{lemma}
\begin{proof}
    This follows from~\cite[Proposition 2.21]{dlHGhys-90}.
\end{proof}

It would be interesting to know if  one can get better quantitative bounds between the hyperbolicity constants than the ones provided by these results, when restricting to the case of $1$-skeleta of cube complexes, as this would improve the upper bound in Theorem \ref{thm:Main-Dehn}.

We conclude this discussion with the definition of hyperbolicity for groups.

\begin{definition}
    A finitely generated group $G$ is \emph{hyperbolic} if it acts properly discontinuously and cocompactly by isometries on a geodesic $\delta$-hyperbolic metric space, for some $\delta\ge0$.
\end{definition}

\subsection{\texorpdfstring{${\rm CAT}(0)$}{CAT(0)} cube complexes and Morse theory}\label{sec:Cat-0-cube-complexes}

A cube complex $X$ is an affine cell complex such that all of its cells are isometric to unit cubes $\left[0,1\right]^n$. It can be endowed with a quotient metric, which restricts to the standard Euclidean metric on every cube. We call $X$ finite-dimensional if there is a uniform upper bound on the dimension of its cubes and locally finite if every $k$-cube only belongs to finitely many $(k+1)$-cubes. We refer to~\cite{BriHae-99} for a detailed introduction to cube complexes. Here we will only summarise the most important results required for our work.

A \emph{height map} (or \emph{Morse function}) $h: X\to \mathbb{R}$ on a cube complex $X$ is a map with the following properties:
\begin{enumerate}
    \item the restriction of $h$ to every cube of $X$ is affine linear; and
    \item the restriction of $h$ to every $1$-cube in $X$ is a homeomorphism onto an interval $\left[n,n+1\right]\subset \mathbb{R}$ for some integer $n\in \mathbb{Z}$.
\end{enumerate}

More generally, we call a map $h:X\to S^1=\left[0,1\right]/\,0\sim 1$ a \emph{height map} (or \emph{Morse function}) if its lift $\widetilde{h}:\widetilde{X}\to \mathbb{R}$ to universal covers is a height map. 

A height map equips the edges of $X$ with a natural orientation induced by the natural orientation on $\mathbb{R}$. Conversely, height maps are often defined by equipping the edges of a cube complex with suitable orientations. 

The link $\lk(v)$ of vertex $v\in X$ is the simplicial complex defined by intersecting $X$ with a sphere of radius $0<\epsilon < 1$. The $m$-cells of $\lk(v)$ are in natural correspondence with $(m+1)$-cubes with vertex $v$; here we allow that a cube contributes more than one $m$-cell if several of its vertices are identified with $v$ in $X$. In particular, the vertices of $\lk(v)$ are in correspondence with the (oriented) edges of $X$ with vertex $v$. 

If $X$ is equipped with a height map, then we use the induced orientation to define the ascending link $\lk^{\uparrow}(v)$ (respectively the descending link $\lk^{\downarrow}(v)$) of a vertex $v\in X$ to be the full subcomplex induced by all vertices of $\lk(v)$ corresponding to outgoing edges (respectively to incoming edges) in $v$.

The main result of Bestvina--Brady's Morse theory~\cite{BesBra-97} is a characterisation of finiteness properties of kernels of homomorphisms induced by Morse functions. Let $X$ be a finite cube complex and let $G=\pi_1(X)$ be its fundamental group. Assume that $h:X\to S^1$ is a Morse function that induces a non-trivial homomorphism $\phi: \pi_1(X)\to \pi_1(S^1)=\mathbb{Z}$ on fundamental groups. 

\begin{theorem}[{\cite[Theorem 4.7]{Bra-99}}]\label{thm:BB-Morse}
    Let $h$ and $\phi$ be as above, let $H:=\ker(\phi)$ and assume that $\widetilde{X}$ is contractible. Then the following hold:
    \begin{enumerate}
        \item if for every vertex $v\in X$ the reduced homology of the ascending and descending links $\lk^{\uparrow}(v)$ and $\lk^{\downarrow}(v)$ is zero over a unital ring $R$ in dimension $\leq n-1$, then $H$ is of type $FP_{n}(R)$;
         \item if for every vertex $v\in X$ the ascending and descending links $\lk^{\uparrow}(v)$ and $\lk^{\downarrow}(v)$ are $(n-1)$-connected, then $H$ is of type $F_{n}$; and
         \item if (1) holds and, in addition, the reduced homology of all ascending and descending links of $X$ is non-zero in dimension $n$ and zero in dimension $n+1$, then $H$ is of type $FP_{n}(R)$ and not of type $FP_{n+1}(R)$.
    \end{enumerate}
\end{theorem}

The finiteness conditions $FP_n(R)$ in Theorem \ref{thm:BB-Morse} are algebraic analogues of the homotopical finiteness properties $F_n$, defined via projective resolutions. We omit their precise definition here, since we will not require it. 

Finally, we call a cube complex $X$ \emph{locally ${\rm CAT}(0)$} if it is finite-dimensional and the link of each of its vertices is a flag complex, and \emph{${\rm CAT}(0)$} if, in addition, $X$ is simply connected. By Gromov's Link Condition, this is equivalent to the classical definition of ${\rm CAT}(0)$ (respectively locally ${\rm CAT}(0)$) in terms of comparison triangles in Euclidean space. Since we will not make direct use of either of the two conditions here, we refer to~\cite{BriHae-99} for further details.

\subsection{Brady's example}\label{sec:Bradys-group}

In~\cite{Bra-99}, Brady constructed the first example of a finitely presented non-hyperbolic subgroup of a hyperbolic group. We summarise his construction, as this will be the main example to which we will apply our methods. 

Let $\Theta ^3$ be a direct product of three copies of the graph $\Theta$ with two vertices and four edges between them. We orient two edges of $\Theta$ in one direction and the other two in the other. This orientation provides a natural map $f: \Theta ^3 \to S^1$ to which one can apply Bestvina--Brady Morse theory. The ascending and descending links of $f$ are octahedra (obtained as joins of three $0$-spheres). Thus, $\ker(f_{\ast}: \pi_1(\Theta^3)\to \pi_1(S^1))$ is of type $F_2$ and not $F_3$. Brady then takes a ramified cover $X\to \Theta^3$ of this product with ramification locus consisting of three disjoint copies of $\Theta$ (one aligned with each factor) and proves that one can still apply Morse theory to the induced map $g: X\to S^1$ ($X$ inherits a natural cubulation from $\Theta^3$). The ascending and descending links are now octahedra or suspensions of a $20$-cycle, depending on whether the corresponding vertex of $X$ lies outside or inside the ramification locus. It thus follows from Bestvina--Brady Morse theory that $\ker(g_{\ast})$ is $F_2$ and not $F_3$. Moreover, Brady proves that $X$ is a $\delta$-hyperbolic space and thus $\pi_1(X)$ is a hyperbolic group, hence producing an example of a non-hyperbolic finitely presented subgroup of a hyperbolic group.

For our proof it will be important to have a precise understanding of the links, as well as the ascending and descending links, of the cube complex $X$ with respect to the map $g$. We thus recall some results from~\cite{Bra-99}. Before stating them, we require a further definition. A \emph{ramified} vertex, edge or subgraph of $X$ (resp. $\tX$) will be a subset of the ramification locus of the cover $X\to \Theta^3$ (resp. $\tX\to \Theta^3$). If a vertex, edge or subgraph is not ramified, we call it \emph{unramified}.

\begin{lemma}\label{lem:4-cycles}
Let $v\in X$ be a vertex. Its link has the following form
\begin{enumerate}
    \item if $v$ is unramified, then $\lk(v)=F\ast F\ast F$ for $F$ a discrete set with $4$ elements;
    \item if $v$ is ramified, then (up to permutation of factors) $\lk(v)=\Gamma\ast F$ for $F$ a discrete set with $4$ elements and $\Gamma\to F\ast F$ a $5$-fold covering such that the preimage of every $4$-cycle in $F\ast F$ in $\Gamma$ is a $20$-cycle. In particular, $\Gamma$ contains no $4$-cycles. 
\end{enumerate}
\end{lemma}
\begin{proof}
    This is proved in~\cite{Bra-99}, after observing that the ``in particular''-part of (2) follows because every embedded $4$-cycle in a graph is $\pi_1$-embedded. Thus, a $4$-cycle in $\Gamma$ would map to an embedded $4$-cycle in $F\ast F$, which would constitute a contradiction. 
\end{proof}

\begin{remark}\label{rmk:4-cycles}
    An analogous result holds for the ascending and descending links of $v$ with respect to $g$, where we replace $F$ by the two-point subsets $F^{\uparrow/\downarrow}:=\lk^{\uparrow/\downarrow}(v) \cap F \subset F$ corresponding to the outgoing/incoming edges, and $\Gamma$ by the preimages $\Gamma^{\uparrow/\downarrow}$ of the subgraphs $F^{\uparrow/\downarrow}\ast F^{\uparrow/\downarrow}$ under the covering map $\Gamma \to F\ast F$. Importantly, the subgraphs $\Gamma^{\uparrow/\downarrow}$ will be connected (in fact they consist of a single $20$-cycle). In particular, all ascending and descending links are spheres of diameter two.
\end{remark}

\begin{remark}
    The graph $\Gamma$ in Lemma~\ref{lem:4-cycles} will contain $k$-cycles with $k<20$. Indeed, it is not hard to see that every $20$-cycle will intersect at least one other $20$-cycle non-trivially. This intersection will create shortcuts between vertices in the $20$-cycle. For examples of cube complexes with larger links such as the ones in~\cite{Lod-18,Kro-21} such intersections will also appear in the ascending and descending links.
\end{remark}

\section{A geometric version of a Theorem of Gersten and Short}
\label{sec:geometric-Gersten-Short}

In this section we will prove a geometric version of a result by Gersten and Short about Dehn functions of cocyclic subgroups of hyperbolic groups~\cite{GerSho-02} that provides area estimates for fillings in level sets of cube complexes with respect to a suitable height function.

\subsection{A theorem of Gersten and Short}
\label{sec:Initial-filling}
We call a subgroup $H$ of a group $G$ \emph{cocyclic} if there is a surjective homomorphism $\phi: G\to \Z$ such that $H=\ker(\phi)$. In~\cite{GerSho-02} Gersten and Short prove the following result.
\begin{theorem}\label{thm:GerSho}
    Let $H\leq G$ be a cocyclic finitely presented subgroup of a hyperbolic group $G$. Then there is a polynomial function $p$ such that $\delta_H(n)\leq p(n)$ for all $n\in \mathbb{N}$.
\end{theorem} 

As explained in the introduction, a careful study of Gersten--Short's proof of Theorem~\ref{thm:GerSho} reveals that the degree of $p$ depends on several seemingly explicit constants, including the hyperbolicity constant $\delta$ in the $\delta$-thinness characterisation of hyperbolicity and a constant measuring the growth of relations of $H$ under the $\Z$-action by conjugation. In theory, one should be able to determine them and deduce an explicit upper bound on the degree of $p$. However, finding presentations for $G$ and $H$ for which this is feasible in practice seems rather challenging.

To circumvent this challenge we will prove a geometric version of Gersten--Short's results for cube complexes. We will say that a loop $\gamma : [0,n]\to Y$ in a CW-complex $Y$, where open edges are identified with bounded open intervals, is a {\em combinatorial loop} if it is an edge loop whose restriction to every edge is parametrised by length. A key ingredient in our proof will be the following geometric version of \cite[Lemma 2.2]{GerSho-02}.

\begin{lemma}\label{lem:hyperbolic-area-radius-pairs}
    Let $Y$ be a compact CW-complex such that the $1$-skeleton $\widetilde{Y}^{(1)}$ of its universal cover $\widetilde{Y}$, equipped with the induced length metric where every edge has length $1$, has $\delta$-thin triangles. Then there are constants $C_1, C_2>0$ such that every combinatorial loop $\gamma:\left[0,n\right]\to \widetilde{Y}$ admits a filling by a van Kampen diagram of area $\leq C_1 n \log_2(n)$ and radius $\leq \delta \log_2(n) + C_2$.
\end{lemma}
\begin{proof}
    The proof is standard and completely analogous to the proof of~\cite[Lemma 2.2]{GerSho-02}; we thus omit the details here. The main idea is to subdivide the loop $\gamma$ into geodesic triangles, such that for every $1\leq j \leq \log_2(n)$ there are $n/2^{\log_2(n)-j}$ triangles with side lengths $\simeq n/2^j$ and area $\simeq n/2^j$, as in Figure~\ref{fig:Farey-graph}. The conclusion then follows readily from the explicit linear area filling for $\delta$-thin triangles constructed in~\cite{GerSho-02}. The constants $C_1$ and $C_2$ are the maximal area and radius of a set of fixed fillings for all loops of length $\leq 3\delta + 3$ in $\widetilde{Y}$, where the compactness of $Y$ ensures their finiteness.
\end{proof}

\begin{figure}
    \begin{center}
    {
    \def\svgwidth{12cm}
    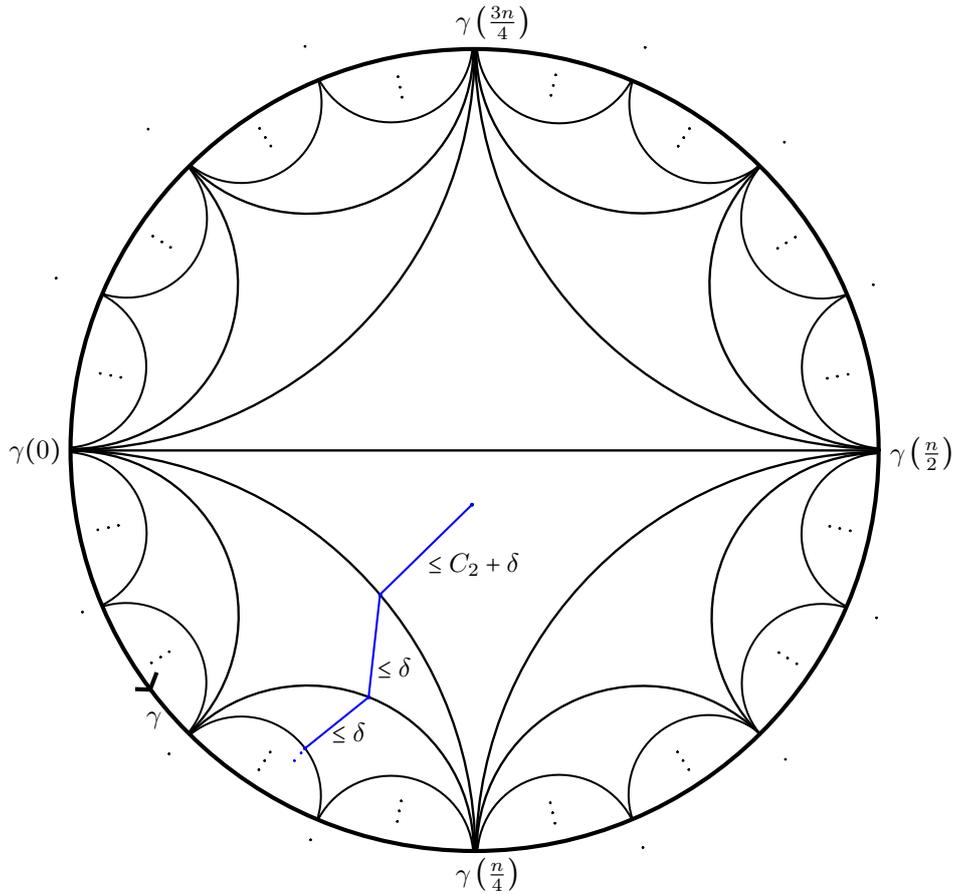}
    \end{center}
    
    \caption{The filling of a loop $\gamma$ in the proof of Lemma~\ref{lem:hyperbolic-area-radius-pairs}}
    \label{fig:Farey-graph}
\end{figure}

\subsection{Fillings in cube complexes}

We will now assume that $G$ is a hyperbolic group which is the fundamental group of a cube complex $X$. We further assume that we have a finitely presented non-hyperbolic subgroup which is the kernel of a surjective homomorphism $\phi : G=\pi_1(X) \to \mathbb{Z}=\pi_1(S^1)$, which is induced by a continuous map $g: X \to S^1$ that maps edges non-trivially onto $S^1$ and restricts to an affine map on the interior of every cube. We denote $h: \tX \to \mathbb{R}$ the lift of $g$ to universal coverings. 

\begin{remark}
    In principle, our approach can also be applied to obtain upper bounds on Dehn functions in the more general setting where we consider a sufficiently nice $\phi$-equivariant height function on a simply connected affine cell complex equipped with a free cocompact $G$-action, where $\phi : G \to \mathbb{Z}$ is some morphism. However, we restrict ourselves to the above situation for simplicity.
\end{remark}

The intersection of cubes with the integer level sets $h^{-1}(m)$, $m\in \Z$, induces a \emph{sliced cell structure} on $\tX$, which restricts to an affine cell structure on $h^{-1}(m)$. This sliced cell structure was introduced in~\cite{CarFor-17} and Figure~\ref{fig:sliced-3-cube} provides an illustration of its restriction to a $3$-cube. In the next section we will push fillings in $\tX$ into the 0-level set $\tX_0$ of $h$. For this it will turn out to be natural to work with this sliced cell structure. On the other hand we will later bound from above the $\delta>0$ such that the $1$-skeleton of $\tX$ equipped with its coarser cubular structure has $\delta$-thin triangles. The following lemma bridges the gap between these two different cell structures on $\tX$.

\begin{lemma}\label{lem:sliced-hyperbolic-area-radius-pairs}
    Assume that the $1$-skeleton of the cube complex $\tX$ equipped with its cubular cell structure has $\delta$-thin triangles. Then, there are constants $C_1,C_2>0$ such that every combinatorial loop $\gamma : [0,n]\to \tX$ in the sliced cell structure on $\tX$ admits a filling by a van Kampen diagram of area $\leq C_1 n \log_2(n)$ and radius $\leq \delta \log_2(n)+C_2$ in the sliced cell structure on $\tX$.
\end{lemma}
\begin{proof}
    By replacing every path along a diagonal edge of a square by a path along two sides of this square, at area cost at most $n$, we can replace $\gamma$ by a combinatorial loop $\widehat{\gamma} : [0,m] \to \tX$ of length $m\leq 2n$ with image in the $1$-skeleton of the cubular cell structure on $\tX$. The assumption that the $1$-skeleton of $\tX$ has $\delta$-thin triangles and \Cref{lem:hyperbolic-area-radius-pairs} imply that there is a van Kampen diagram for $\widehat{\gamma}$ in the cubular cell structure on $\tX$ of area $\leq C_1 n \log_2(n)$ and radius $\leq \delta \log_2(n) + C_2$ for suitable constants $C_1,~C_2>0$. Observe that the edges of the cubular cell structure on $\tX$ are naturally a subset of the edges of the sliced cell structure and every square in the cubular cell structure is subdivided into two triangles in the sliced cell structure. Thus, we can construct from the van Kampen diagram for $\widehat{\gamma}$ in the cubular cell structure a van Kampen diagram for $\gamma$ in the sliced cell structure of area $\leq 2(C_1+1) n\log_2(n)$ and radius $\leq \delta \log_2(n) + C_2 + 1$. Thus, replacing the constants $C_1,~C_2>0$ from \Cref{lem:hyperbolic-area-radius-pairs} by $2(C_1+1)$ and $C_2+1$ completes the proof.
\end{proof}

\subsection{Pushing fillings into level sets}

We now adapt the proof of Theorem~\ref{thm:GerSho} to give a geometric version of it. We will retain the notation from \Cref{sec:Initial-filling}. In particular, $\tX$ will be a cube complex equipped with a Morse function $h:\tX \to \mathbb{R}$ and $\gamma:\left[0,n\right]\to \tX_0$ is a combinatorial loop of length $n$ in the $1$-skeleton of $\tX_0$. Throughout this section we will assume that $\tX$ is equipped with the sliced cell structure induced by $h$, unless explicitly mentioned otherwise. We will further assume that the ascending and descending links of $\tX$ with respect to $h$ are simply connected. Denote by $T$ the maximal filling area of a loop of length $$\leq 2\max\left\{{\rm diam}(\lk^{\downarrow}(v)), {\rm diam}(\lk^{\uparrow}(v))\mid v\in \tX^{(0)}\right\} +1$$ inside these links. The following is the main result of this section.

\begin{proposition}
    \label{prop:geometric-Gersten-Short}
    Let $\tX$ be a cube complex equipped with a Morse function $h:\tX\to \mathbb{R}$ with all ascending and descending links simply connected. Let $C_3:=3 T +1$ and assume that $\tX^{(1)}$ equipped with the cubular cell structure has $\delta$-thin triangles for some $\delta>0$. Then there is a constant $C_4>0$ that only depends on $\tX$ and $h$ such that the area of every combinatorial loop $\gamma:\left[0,n\right]\to \tX_0^{(1)}$ of length $n$ satisfies
    \[
        \Area_{\tX_0}(\gamma)\leq C_4\cdot n^{1+\delta\cdot \log_2(C_3)} \cdot \log_2(n).
    \]
\end{proposition}

\begin{proof}
By Lemma~\ref{lem:sliced-hyperbolic-area-radius-pairs} there is a filling of $\gamma$ (in the sliced cell structure) by a van Kampen diagram $\alpha: D\to \tX^{(2)}$ of area $\leq C_1 n\log_2(n)$ and radius $\leq \delta \log_2(n) + C_2$. The radius bound means that $$h(\alpha(D))\subseteq \left[-\delta\log_2(n)-C_2,\delta\log_2(n) + C_2\right].$$ We call $m:=\max\bigl\{\,|h(\alpha(x))|\,\big\mid\, x\in D\,\bigr\}\leq \delta \log_2(n)+C_2$ the \emph{height of the filling} $\alpha$. 

We will inductively ``push'' this filling down into the 0-level set, reducing its height by 1 in every step. Our inductive step can be interpreted as a geometric version of the inductive reduction of $t$-rings in Gersten--Short's work~\cite{GerSho-02}. However, our actual proof relies on different arguments from the ones in their work.

Our inductive assumption is that after $k$ steps we have reduced to a filling $\alpha_k: D_k\to \tX^{(2)}$ of area $\leq C_3^k\cdot C_1'\cdot n\log_2(n)$ and height $m-k$, where $C'_1\geq C_1$ is a constant that incorporates the cost of extending our pushing map to a $1$-neighbourhood of the connected components at height $m-k$. Note that the case $k=0$ is satisfied by assumption. Since $m\leq \delta \log_2(n) + C_2$, this will imply that
\begin{equation}\label{eqn:upper-area-bound}
	\Area_{\tX_0}(\gamma)\leq C_3^{\delta\log_2(n)+C_2}\cdot C_1' \cdot n\log_2(n)\leq C_4 n^{1+\delta \log_2(C_3)}\log_2(n),
\end{equation}
for $C_4:= C_3^{C_2}\cdot C_1'$, completing the proof.

We explain the inductive procedure for reducing the height of the closed $1$-neighbourhood $N_1(C)$ of a connected component $C$ of $\alpha_k^{-1}(h^{-1}(\pm(m-k)))$ by one. To simplify notation, we will not distinguish in notation between $C$ and its image in the $\pm (m-k)$-level set of $\tX$. However, we emphasise that our line of argument actually inductively replaces the van Kampen diagram $D$ by a van Kampen diagram of height $\leq m-k$, which is in general only immersed in $\tX$. Moreover, we will now assume that $C$ maps into the $(m-k)$-level set. The case when $C$ maps into the $-(m-k)$-level set is analogous, except that one needs to replace descending links by ascending links everywhere.

Observe that the boundary $\partial N_1(C)$ of the closed $1$-neighbourhood of $C$ is contained in the $(m-k-1)$-level set of $h$. Our goal is to replace the filling $\alpha_k|_{N_1(C)}$ by a filling $\beta:C'\to \tX^{(2)}$ of height $\leq m-k-1$ with $\beta|_{\partial C'} = \alpha|_{\partial N_1(C)}$. Doing this with all connected components of $\alpha_k^{-1}(h^{-1}(\pm(m-k)))$ will yield a filling $\alpha_{k+1}$ of height $m-k-1$ that will satisfy the asserted area bound.

To construct the filling $\beta$ we first construct a homotopy of the restriction of $\alpha_k$ to the $1$-skeleton of $C$ into the $(m-k-1)$-level set, starting with the vertices, then we extend this homotopy to the $2$-skeleton, and finally we extend it to $N_1(C)$. After equipping $N_1(C)$ with a suitable cell structure, which is implicit in the construction of the homotopy, this produces the desired filling $\beta$.

\begin{figure}[!ht]
\begin{tikzpicture}[scale=2] 
\fill (1,0) circle (1pt);
\fill (1,2) circle (1pt);
\fill (0,1) circle (1pt) node [above left=1pt] {$v$};
\fill (2,1) circle (1pt);
\draw[->-=0.8] (1,0)--(0,1);
\draw[->-=0.8] (1,0)--(2,1);
\draw[->-=0.6] (0,1)--(1,2);
\draw[->-=0.6] (2,1)--(1,2);
\draw (0,1)--node[above=1pt] {$e$}(2,1);
\draw[->-=0.8] (-0.5,0.15)--(0,1);
\draw [red,thick] (0.51,0.49)--(1.49,0.49);
\draw [red,thick,snake it] (-0.3,0.49) arc (180:360:0.4cm and 0.03cm);
\draw [red] (0.1,0.2) node {$\mu_{(v,e)}$};
\fill (-0.3,0.49) circle (1pt) node [left=2pt] {$w_v$};
\fill (0.51,0.49) circle (1pt);
\draw (0.7,0.6) node {$w_{(v,e)}$};
\fill (1.49,0.49) circle (1pt);

\draw (-0.75,1.25) node {$C:$};
\draw[mapping=0.99] (2.75,1)--(3.75,1);

\begin{scope}[xshift=-3.5cm]
\draw (8,2.5) node[right=1pt]{$\R$};
\draw[->] (8,-0.5)--(8,2.5);
\fill (8,0) circle (1pt) node[right=1pt] {$m-k-1$};
\fill (8,1) circle (1pt) node[right=1pt] {$m-k$};
\fill (8,2) circle (1pt) node[right=1pt] {$m-k+1$};
\end{scope}
\end{tikzpicture}
    \caption{Pushing the $1$-skeleton around a vertex $v\in C$.}
  \label{fig:locally-pushing-the-1-skeleton}
\end{figure}
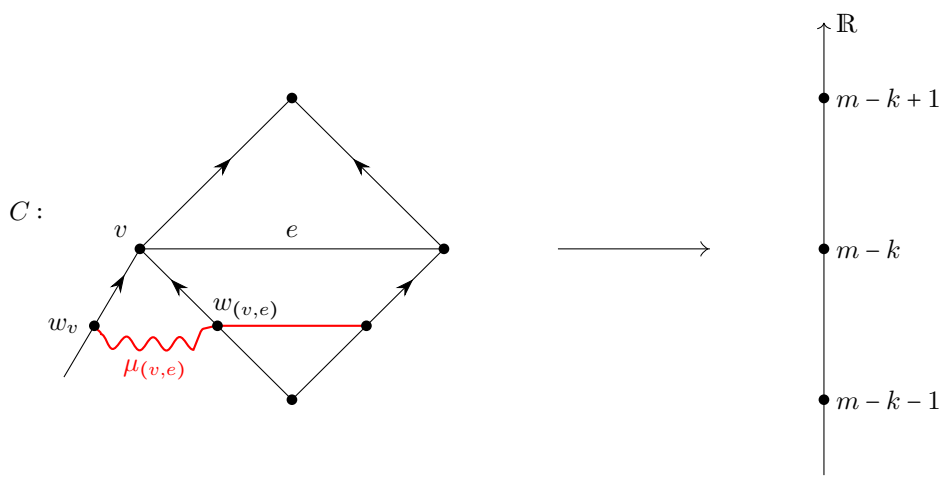

\smallskip
\noindent \emph{Step 1 (Pushing down the $1$-skeleton of $C$).}

We begin by choosing for every vertex $v\in C$ an arbitrary vertex $w_v\in \lk^{\downarrow}(v)$ in its descending link; the vertex $w_v$ corresponds to a choice of an incoming edge in $v$. By definition of the sliced cell structure on $\tX$, every edge $e\in C$ corresponds to the diagonal of a square in $\tX$. We can thus associate to a pair $(v,e)$ consisting of a vertex $v\in C$ and an edge $e\in C$ which has $v$ as one of its endpoints the unique vertex $w_{(v,e)}\in \lk^{\downarrow}(v)$ corresponding to the edge of this square which connects $v$ to the $(m-k-1)$-level set. We fix a geodesic $\mu_{(v,e)}$ between $w_v$ and $w_{(v,e)}$ in $\lk^{\downarrow}(v)$ (see Figure~\ref{fig:locally-pushing-the-1-skeleton}). Then the geodesics $\mu_{(v,e)}$ together with affine homotopies of the edges $e$ inside their corresponding squares define a homotopy of the restriction of $\alpha_k$ to the $1$-skeleton of $C$ to a map with image in the $1$-skeleton of $\tX$ at height $(m-k-1)$. More precisely, this map maps vertices $v\in C$ to the endpoints of the edges defined by the $w_v$ and edges $e=\left[v_1,v_2\right]\in C$ to the paths induced by the concatenations $\mu_{(v_1,e)}\cdot \mu_{(v_2,e)}^{-1}$ in the $(m-k-1)$-level set.

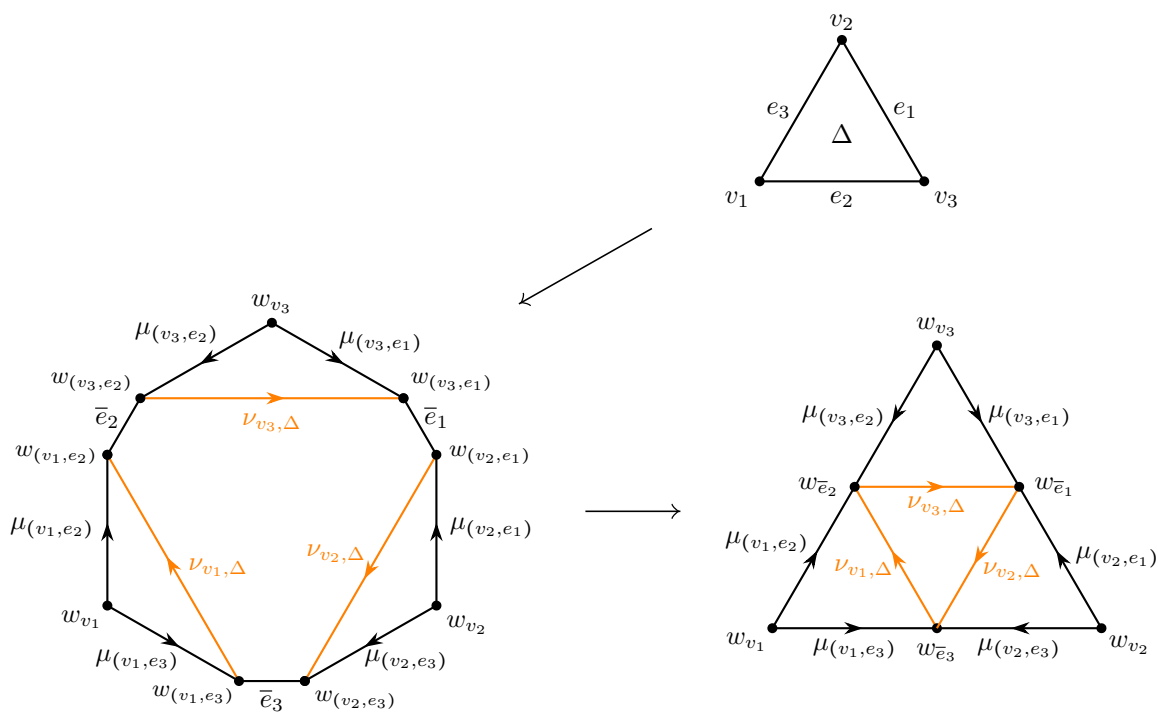
\begin{figure}
    \centering
		\tikzmath{\x1=0.2;}

		\begin{tikzpicture}[thick,scale=2.5]
				
				
			\coordinate (A) at (0.866*\x1, -1+0.5*\x1);
			\coordinate (B) at (-0.866*\x1, -1+0.5*\x1);
			\coordinate (C) at (-0.866, 0.5-\x1);

\begin{scope} 
			\draw (A)--(B);
			\draw [->-=0.55] (-0.866,-0.5)--(B);
			\draw [->-=0.55] (-0.866,-0.5)--(C);
			\draw [orange, ->-=0.55] (B)-- node[right=2pt] {$\nu_{v_1,\Delta}$} (C);
			\fill (A) circle (0.75pt);
			\fill (B) circle (0.75pt);
			\fill (C) circle (0.75pt);
			\fill (-0.866,-0.5) circle (0.75pt);
			
	\begin{scope}[rotate=-120]
			\draw ([rotate=-120]A)--([rotate=-120]B);
			\draw [->-=0.55] (-0.866,-0.5)--([rotate=-120]B);
			\draw [->-=0.55] (-0.866,-0.5)--([rotate=-120]C);
			\draw [orange, ->-=0.55] ([rotate=-120]B)-- node[below=2pt] {$\nu_{v_3,\Delta}$} ([rotate=-120]C);
			\fill ([rotate=-120]A) circle (0.75pt);
			\fill ([rotate=-120]B) circle (0.75pt);
			\fill ([rotate=-120]C) circle (0.75pt);
			\fill (-0.866,-0.5) circle (0.75pt);
	\end{scope}		

	\begin{scope}[rotate=120]
			\draw ([rotate=120]A)--([rotate=120]B);
			\draw [->-=0.55] (-0.866,-0.5)--([rotate=120]B);
			\draw [->-=0.55] (-0.866,-0.5)--([rotate=120]C);
			\draw [orange, ->-=0.55] ([rotate=120]B)-- node[above left=-2pt] {$\nu_{v_2,\Delta}$} ([rotate=120]C);
			\fill ([rotate=120]A) circle (0.75pt);
			\fill ([rotate=120]B) circle (0.75pt);
			\fill ([rotate=120]C) circle (0.75pt);
			\fill (-0.866,-0.5) circle (0.75pt);
	\end{scope}		

		\draw (0,-1) node {$\overline e_3$}; 
		\draw (-0.866,0.5) node {$\overline e_2$};
		\draw (0.866,0.5) node {$\overline e_1$};
		
		\draw (0,1) node[above] {$w_{v_3}$};
		\draw (0.866,-0.5) node[below right] {$w_{v_2}$};
		\draw (-0.866,-0.5) node[below left=-2pt] {$w_{v_1}$};
		
		\draw ([rotate=-120]B) node[above left=-1pt] {\small$w_{(v_3,e_2)}$};
		\draw ([rotate=120]A) node[above right=-1pt] {\small$w_{(v_3,e_1)}$};
		\draw ([rotate=-120]A) node[left] {\small$w_{(v_1,e_2)}$};
		\draw ([rotate=120]B) node[right=1pt] {\small$w_{(v_2,e_1)}$};
		\draw (B) node[below left=-2pt] {\small$w_{(v_1,e_3)}$};
		\draw (A) node[below right] {\small$w_{(v_2,e_3)}$};
		
		\draw ([rotate=30]0,1) node[above=-2pt] {$\mu_{(v_3,e_2)}$};
		\draw ([rotate=-35]0,1) node[above=-2pt] {$\mu_{(v_3,e_1)}$};
		\draw ([rotate=45]0,-1) node[below=-2pt] {$\mu_{(v_2,e_3)}$};
		\draw ([rotate=-45]0,-1) node[below=-2pt] {$\mu_{(v_1,e_3)}$};
		\draw ([rotate=5]-1,0) node[left=-8pt] {$\mu_{(v_1,e_2)}$};
		\draw ([rotate=-5]1,0) node[right=-8pt] {$\mu_{(v_2,e_1)}$};
\end{scope} 

\draw[semithick,mapping=0.99] (1.65,0)--(2.15,0); 

\begin{scope}[xshift=3.5cm, yshift=-0.12cm]  
			\draw [->-=0.55] (-0.866,-0.5)-- node[below] {$\mu_{(v_1,e_3)}$}(0,-0.5);
			\draw [->-=0.55] (-0.866,-0.5)-- node[above left=-2pt] {$\mu_{(v_1,e_2)}$}(-0.433,0.25);
			\draw [orange, ->-=0.55] (0,-0.5)-- node[below left=-3pt] {$\nu_{v_1,\Delta}$} (-0.433,0.25);
			\fill (-0.866,-0.5) circle (0.75pt) node [below left=-2pt] {$w_{v_1}$};
			\fill (0,-0.5) circle (0.75pt) node [below=2pt] {$w_{\overline e_3}$};
			
		\begin{scope}[rotate=-120]
			\draw [->-=0.55] (-0.866,-0.5)--node[left] {$\mu_{(v_3,e_2)}$}(0,-0.5);
			\draw [->-=0.55] (-0.866,-0.5)--node[right] {$\mu_{(v_3,e_1)}$}(-0.433,0.25);
			\draw [orange, ->-=0.55] (0,-0.5)-- node[below] {$\nu_{v_3,\Delta}$} (-0.433,0.25);
			\fill (-0.866,-0.5) circle (0.75pt) node [above] {$w_{v_3}$};
			\fill (0,-0.5) circle (0.75pt) node [left=2pt] {$w_{\overline e_2}$};
		\end{scope}
		
		\begin{scope}[rotate=120]
			\draw [->-=0.55] (-0.866,-0.5)--node[right=2pt] {$\mu_{(v_2,e_1)}$}(0,-0.5);
			\draw [->-=0.55] (-0.866,-0.5)--node[below] {$\mu_{(v_2,e_3)}$}(-0.433,0.25);
			\draw [orange, ->-=0.55] (0,-0.5)-- node[below right=-2pt] {$\nu_{v_2,\Delta}$} (-0.433,0.25);
			\fill (-0.866,-0.5) circle (0.75pt) node [below right=-1pt] {$w_{v_2}$};
			\fill (0,-0.5) circle (0.75pt) node [right=2pt] {$w_{\overline e_1}$};
		\end{scope}
\end{scope} 
		
\draw[semithick,mapping=0.99] (2,1.5)--(1.3,1.1); 
				
\begin{scope}[xshift=3cm, yshift=2cm]		

	\draw (-0.433,-0.25)-- node[left] {$e_3$}(0,0.5)--node[right]{$e_1$}(0.433,-0.25)--node[below]{$e_2$}(-0.433,-0.25);
	\fill (-0.433,-0.25) circle (0.75pt) node [below left] {$v_1$};
	\fill (0.433,-0.25) circle (0.75pt) node [below right] {$v_3$};	
	\fill (0,0.5) circle (0.75pt) node [above] {$v_2$};
	\draw (0,0) node {$\Delta$};		
		
\end{scope} 
		
\end{tikzpicture}

    \caption{Pushing a $2$-cell $\Delta$ in $C$ into the $(m-k-1)$-level set}
    \label{fig:pushing-2-cells}
\end{figure}

\smallskip
\noindent \emph{Step 2 (Pushing down the $2$-cells of $C$).}

We need to extend the above homotopy to the $2$-cells of $C$. For this we observe that the $2$-cells are triangles arising by slicing a $3$-cube along a level set. They are determined by their three vertices $v_1$, $v_2$ and $v_3$ which lie in a common $3$-cube. For each of the $v_i$ we denote by $\nu_{v_i,\Delta}$ a geodesic in $\lk^{\downarrow}(v_i)$ between $w_{(v_i,e_j)}$ and $w_{(v_i,e_k)}$ for $\left\{i,j,k\right\}=\left\{1,2,3\right\}$ the labelling in Figure~\ref{fig:pushing-2-cells}. Since $\Delta$ is contained in a $3$-cube the $\nu_{v_i,\Delta}$ either have length $1$ (if $\Delta$ is a top triangle) or length $0$ (if $\Delta$ is a bottom triangle); see Figure~\ref{fig:sliced-3-cube}. 

Since all ascending and descending links of $\tX$ are simply connected, we can fill the loop $\mu_{(v_1,e_3)}\cdot \nu_{v_1,\Delta}\cdot \mu_{(v_1,e_2)}^{-1}$ in $\lk^{\downarrow}(v_1)$ and analogously for the loops based at $w_{v_2}$ and $w_{v_3}$. Moreover, the $\nu_{v_i,\Delta}$ are either all constant or they span a triangle in the sliced cell structure of the $(m-k-1)$-level set. This provides us with a way of extending our homotopy of the restriction of $\alpha_k$ to the $1$-skeleton of $C$ to a homotopy of its restriction to $C$ so that the resulting map has image in the $(m-k-1)$-level set. Moreover, every $2$-cell in $C$ is replaced by at most
\begin{multline}\label{eqn:area-of-replacing-a-2-cell}
    1 + \Area_{\lk^{\downarrow}(v_1)}(\mu_{(v_1,e_3)}\cdot \nu_{v_1,\Delta}\cdot \mu_{(v_1,e_2)}^{-1}) + \Area_{\lk^{\downarrow}(v_2)}(\mu_{(v_2,e_1)}\cdot \nu_{v_2,\Delta}\cdot \mu_{(v_2,e_3)}^{-1}) + \\
    \Area_{\lk^{\downarrow}(v_3)}(\mu_{(v_3,e_2)}\cdot \nu_{v_3,\Delta}\cdot \mu_{(v_3,e_1)}^{-1})
\end{multline}
$2$-cells. Since geodesics in $\lk^{\downarrow}(v_i)$ have length at most ${\rm diam}(\lk^{\downarrow}(v_i))$ and the $\nu_{v_i,\Delta}$ have length at most one, we deduce that the expression~\eqref{eqn:area-of-replacing-a-2-cell} is bounded above by $C_3=3T+1$. This implies that the area of the filling obtained by the homotopy that pushes the image of $C$ into the level set at height $m-k-1$ is $\leq C_3 \cdot \Area(C)$.

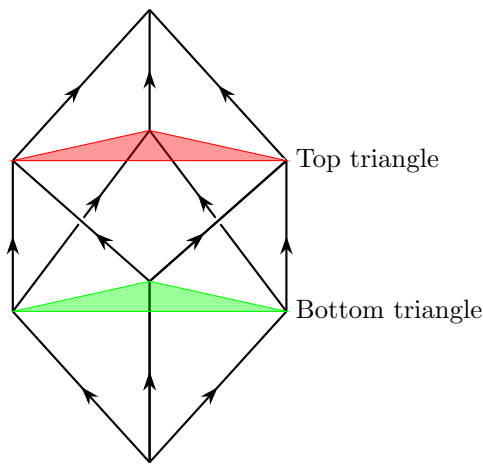
\begin{figure}
    \centering
    \begin{tikzpicture}

    \coordinate (A) at (0,0); 
    \coordinate (B) at (-1.8,2);
    \coordinate (C) at (0,4.4);
    \coordinate (D) at (1.8,2);
    \coordinate (E) at (0,2.4); 
    \coordinate (F) at (-1.8,4);
    \coordinate (G) at (0,6);
    \coordinate (H) at (1.8,4);

    \draw[thick, ->-=0.5] (A) -- (B);
    \draw[thick, ->-=0.65] (B) -- (C);
    \draw[thick, ->-=0.65] (D) -- (C);
    \draw[thick, ->-=0.5] (A) -- (D);

    \draw[thick, ->-=0.4] (E) -- (F);
    
    \draw[thick, ->-=0.5] (H) -- (G);

    \draw[thick, ->-=0.5] (A) -- (E);
    
    \draw[thick, ->-=0.5] (C) -- (G);

    \draw[white, line width=3pt] (E) -- (F);
    \draw[white, line width=3pt] (E) -- (H);
    
    \draw[thick, ->-=0.4] (E) -- (F);
    \draw[thick, ->-=0.4] (E) -- (H);

    \draw[thick, ->-=0.5] (B) -- (F);
    \draw[thick, ->-=0.5] (F) -- (G);
    \draw[thick, ->-=0.5] (A) -- (E);
    \draw[thick, ->-=0.4] (E) -- (H);
    \draw[thick, ->-=0.5] (D) -- (H);
    
    \fill[fill=red, draw=red, fill opacity=0.4] (C)--(F)--(H)--(C);
    \fill[fill=green, draw=green, fill opacity=0.4] (B)--(E)--(D)--(B);

    \node at (D) [right] {Bottom triangle};
    \node at (H) [right] {Top triangle};

\end{tikzpicture}
    \caption{The sliced cell structure on a $3$-cube with coloured top and bottom triangles}
    \label{fig:sliced-3-cube}
\end{figure}

\smallskip
\noindent \emph{Step 3 (Extending the filling from $C$ to $N_1(C)$).}

To construct the filling $\beta$ and prove \eqref{eqn:upper-area-bound}, it suffices to explain that we can extend the van Kampen diagram that we constructed in Step 2 to a van Kampen diagram $\beta$ which coincides with $\alpha_k$ in $\partial N_1(C)$ and still satisfies the asserted area bounds. It suffices to explain how to extend the homotopy of $C$ to each of the connected components of $N_1(C)\setminus C$. Since $C$ is a planar $2$-complex, its complement inside the plane is homeomorphic to a disjoint union of open discs and an unbounded component. Moreover, these components are in correspondence with the components of $N_1(C)\setminus C$. We fix such a component, take its closure and parametrise its boundary that intersects $\partial C$ non-trivially by a circle $K$. The cell structure on $\partial C$ induces a cell structure on $K$. By the definition of the sliced cell structure on $\tX$ at every vertex of $K$ there will be a finite (non-trivial) number of edges connecting it to the $(m-k-1)$-level set inside $N_1(C)$ and every edge of $K$ will be the diagonal of a square with a unique vertex in the $(m-k-1)$-level set. This means that locally around a vertex of $K$ the closure of the connected component of $N_1(C)\setminus C$ corresponding to $K$ looks as in Figure~\ref{fig:neighbourhood-of-boundary-vertex}. In particular, there is a canonical homotopy of $K$ into the $(m-k-1)$-skeleton inside its connected component in $N_1(C)\setminus C$, which is indicated in green and red in Figure~\ref{fig:neighbourhood-of-boundary-vertex}, where $\eta_{(e,f)}$ is the path defined inside $\lk^{\downarrow}(v_2)$ connecting the vertices $w_{(v_2,e)}$ and $w_{(v_2,f)}$ via the sequence of descending edges starting at $v_2$. We can thus extend the homotopy of $C$ to a homotopy of $N_1(C)$ that fixes the boundary as indicated in Figure~\ref{fig:homotopy-extension-to-boundary}.

\begin{figure}
    \centering
    \begin{tikzpicture}[scale=1.5]
		\begin{scope}[blue]
			\fill (0,0) circle (1.5pt); 
			\fill (2,0) circle (1.5pt) node [below=4pt] {$U_e$}; 
			\fill (4,0) circle (1.5pt) node [below=4pt] {$U_f$}; 
			\fill (6,0) circle (1.5pt);
			\draw[thick] (0,0)--(1,1.5)--(2,0)--(3,1.5)--(4,0)--(5,1.5)--(6,0)--(0,0);
			\fill (2.8,0) circle (1.5pt); \fill (3.4,0) circle (1.5pt);
			\draw [thick] (3,1.5)--(2.8,0); \draw [thick] (3,1.5)--(3.4,0);
			\draw [thick,dashed] (1,1.5)--(0.8,0); \draw [thick,dashed] (1,1.5)--(1.4,0);
			\draw [thick,dashed] (5,1.5)--(4.8,0); \draw [thick,dashed] (5,1.5)--(5.4,0);
			\draw (7.5,0.7) node {$N_1(C)\setminus C$};
		\end{scope}
		
		\begin{scope}[black]
		\fill (1,1.5) circle (1.5pt) node [above=2pt] {$v_1$};
		\fill (3,1.5) circle (1.5pt) node [above=2pt] {$v_2$}; 
		\fill (5,1.5) circle (1.5pt) node [above=2pt] {$v_3$};
		\draw[very thick] (1,1.5)--node[above=1pt]{$e$}(3,1.5)--node[above=1pt]{$f$}(5,1.5);
		\draw[dashed,thick] (-0.5,1.5)--(1,1.5);
		\draw[dashed,thick] (5,1.5)--(6.5,1.5); \draw (7,1.5) node {$K$};
		\end{scope}
		
		\begin{scope}[green]
		\fill (2.9,0.75) circle (1.25pt); \fill (3.2,0.75) circle (1.25pt);
		\draw [dashed,thick] (0.5,0.75)--(1.5,0.75);
		\draw [thick] (2.5,0.75)--(3.5,0.75);
		\draw [dashed,thick] (4.5,0.75)--(5.5,0.75);
		\draw (3.0,0.5) node {$\eta(e,f)$};
		\end{scope}
		
		\begin{scope}[red]
			\fill (1.5,0.75) circle (1.5pt);
			\fill (2.5,0.75) circle (1.5pt);
			\fill (3.5,0.75) circle (1.5pt);
			\fill (4.5,0.75) circle (1.5pt);
			\draw[thick] (1.5,0.75)--node[above=2pt]{$\overline e$}(2.5,0.75);
			\draw[thick] (3.5,0.75)--node[above=2pt]{$\overline f$}(4.5,0.75);
		\end{scope}
			
    \end{tikzpicture}

    \caption{A connected neighbourhood of a boundary vertex inside $N_1(C)\setminus C$}
    \label{fig:neighbourhood-of-boundary-vertex}
\end{figure}
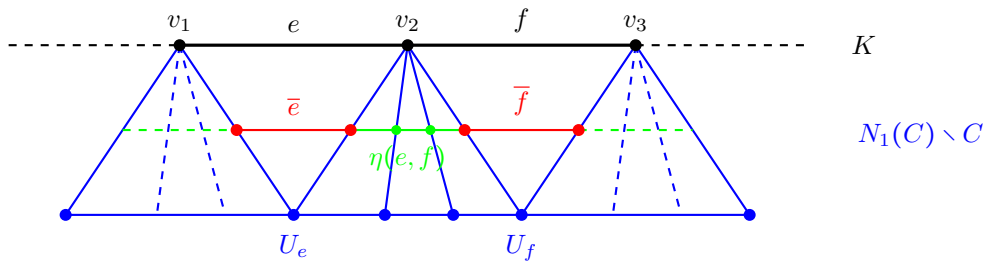

\begin{figure}
    \centering
    \begin{tikzpicture}[scale=2]
    \coordinate (Ue) at (0, 0);
    \coordinate (Uf) at (4, 0);
    \coordinate (V1) at (-1, 1);
    \coordinate (V2) at (2, 1);
    \coordinate (V3) at (5, 1);

    \draw[thick, green] (Ue) -- node[below, green] {$\eta(e,f)$} (Uf);
    
    \draw[thick, red, ->-=0.5] (V1) -- node[pos=0.4, below, sloped] {$\mu_{c(v_1, e)}$} (Ue);
    \draw[thick, red, ->-=0.5] (V2) -- node[pos=0.5, below, sloped] {$\mu_{c(v_2, e)}$} (Ue);
    \draw[thick, green, -] (Ue) -- (Uf);
    \draw[thick, red, ->-=0.5] (V2) -- node[pos=0.5, below, sloped] {$\mu_{c(v_2, f)}$} (Uf);
    \draw[thick, red, ->-=0.5] (V3) -- node[pos=0.5, below, sloped] {$\mu_{c(v_3, e)}$} (Uf);
    
    \filldraw[red] (V1) circle (1.25pt) node[left] {$W_{v_1}$};
    \filldraw[red] (V2) circle (1.25pt) node[above] {$W_{v_2}$};
    \filldraw[red] (V3) circle (1.25pt) node[right] {$W_{v_3}$};
    \filldraw[red] (Ue) circle (1.25pt) node[below left, blue] {$U_E$};
    \filldraw[red] (Uf) circle (1.25pt) node[below right, blue] {$U_F$};
    
\end{tikzpicture}
    \caption{Local extension of the homotopy of $C$ to a connected component of $N_1(C)\setminus C$}
    \label{fig:homotopy-extension-to-boundary}
\end{figure}
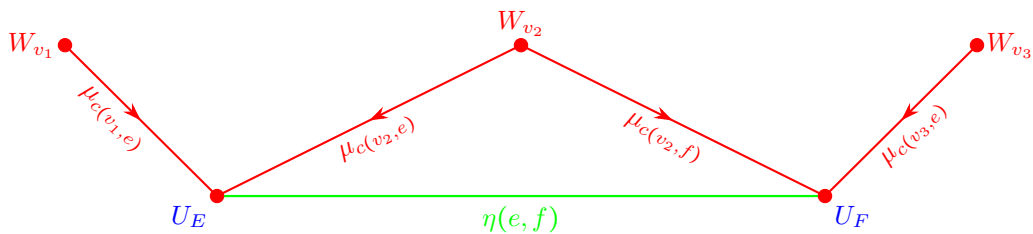

We observe that our extension of the filling to $N_1(C)$ required additional horizontal $2$-cells. However, every edge and vertex of the initial filling $\alpha$ will only create such $2$-cells at most once when we convert its adjacent non-horizontal $2$-cells going from height $m-k$ to height $m-k-1$ into horizontal $2$-cells. Moreover, the horizontal $2$-cells created will only start expanding by $C_3$ in the subsequent pushing steps. The number of new horizontal $2$-cells created in this way is bounded above by the super-additive closure of the maximum of the filling functions of the ascending and descending links, evaluated in the number of $2$-cells of the filling $\alpha$. Since there is a finite number of isomorphism types of ascending and descending links and their filling functions are linear, we deduce that the total number of $2$-cells created in this way will be bounded above by a constant multiple of the initial area of our filling. In particular, this can be accounted for by replacing the constant $C_1$ in the initial area bound for $\alpha$ by a sufficiently large constant $C_1'\geq C_1$ that only depends on $X$ and $h$. This completes the induction step and thus the proof of the inequality~\eqref{eqn:upper-area-bound}.
\end{proof}

\begin{remark}
    \Cref{prop:geometric-Gersten-Short} only gives information in the case that $T$ is finite. This is for instance the case if either $\tX$ admits a cocompact group action or there is a uniform bound on the diameter and filling functions of the ascending and descending links in $\tX$.     
\end{remark}

\begin{corollary}\label{cor:geometric-Gersten-Short}
    With the same assumptions as in \Cref{prop:geometric-Gersten-Short}, assume that there is a group $H$ that acts properly cocompactly on $\tX_0$. Then $\delta_H(n)\preccurlyeq n^{1+\delta\cdot \log_2(C_3)}\cdot \log_2(n)$.
\end{corollary}
\begin{proof}
    This is a direct consequence of \Cref{prop:geometric-Gersten-Short} and \Cref{cor:Alonso}.
\end{proof}

\begin{theorem}\label{thm:Dehn-Bradys-group}
 If $H$ is Brady's group from~\cite{Bra-99}, $X$ is the associated cube complex, $g:X\to S^1$ the associated height function, and $h\colon \tX\to \R$ its lift to the universal cover, then we can choose $C_3=61$ and $\delta=16$, and obtain $\delta_H(n)\preccurlyeq n^{96}$.
\end{theorem}
\begin{proof}
 To obtain an upper bound on $C_3$, observe that by Remark~\ref{rmk:4-cycles} all ascending and descending links for Brady's cube complex are spheres of diameter two, which are obtained by taking the suspension of a $20$-cycle in the ramified case and a $4$-cycle in the unramified case. In particular, we obtain an upper bound on $T$ by giving an upper bound on the filling area of a cycle of length at most $5$ in such a sphere. It is easy to see that any such cycle with non-trivial area will be simple. This implies that $T\leq 20$.\footnote{It is easy to give an example of a $4$-cycle of area precisely 20, meaning that with our approach we can not do better than this.} We deduce that we can choose $C_3=3T+1=61$. To complete the proof we need to show that we can choose $\delta=16$. This follows from \Cref{cor:hyperbolicity-constant}, which we will prove in \Cref{sec:upper-bound-on-hyperbolicity-constant}.
\end{proof}

\begin{proof}[Proof of \Cref{thm:Main-Dehn}]
The result is now a direct consequence of \Cref{thm:Dehn-Bradys-group}.
\end{proof}

\begin{remark}
    The constant $C_3$ in \Cref{prop:geometric-Gersten-Short} captures the worst possible area expansion of a $2$-cell. By taking into account that the ascending and descending links at different vertices may be different, a more careful analysis would improve this bound for many concrete examples. For instance, in Brady's  cube complex only a certain fraction of the vertices is ramified and a more careful analysis should lead to an improvement in the upper bound on its Dehn function. However, giving such an improved bound would be more technical and probably would still not give an optimal bound. We therefore do not pursue this line of argument here.
\end{remark}

\section{An upper bound on the hyperbolicity constant \texorpdfstring{$\delta$}{delta}}\label{sec:upper-bound-on-hyperbolicity-constant}
In this section we will prove \Cref{thm:Main-hyperbolicity-constant}. As a consequence, we will obtain an upper bound on the $\delta>0$ such that the $1$-skeleton of the cube complex $\tX$ from Section~\ref{sec:Bradys-group} equipped with its cubular cell structure has $\delta$-thin triangles. This will complete the proof of \Cref{thm:Main-Dehn}. Throughout this section, we will assume that cube complexes are equipped with their cubical cell structure, unless mentioned otherwise. We will further assume that all graphs have edges of side length 1 and are endowed with the induced length metric. 

\subsection{A quantitative hyperbolicity criterion for cube complexes}

The remaining step for completing the proof of \Cref{thm:Dehn-Bradys-group} and thus of \Cref{thm:Main-Dehn} is to determine an explicit constant $\delta\geq 0$ such that the graph $\tX^{(1)}$ defined by the $1$-skeleton of the cube complex from Section~\ref{sec:Bradys-group} has $\delta$-thin triangles. To do so, we will combine the following two results on median graphs with properties of the cube complex $\tX$ from \Cref{sec:Bradys-group}. This will also yield a proof of \Cref{thm:Main-hyperbolicity-constant}. Since we will not require the definition of a median graph here, we omit it and refer the interested reader to~\cite{Che-00, CDEHV-08}.

\begin{theorem}[Chepoi~\cite{Che-00}]\label{thm-Chepoi}
	Let $\tX$ be a ${\rm CAT}(0)$ cube complex. Then $\tX^{(1)}$ is a median graph. Conversely, every median graph is the $1$-skeleton of a ${\rm CAT}(0)$ cube complex, which is unique up to isomorphisms of cube complexes.
\end{theorem}

A graph is called a \emph{$(C\times D)$-rectangle} if it is isomorphic to the full subgraph contained in a rectangle of side lengths $C,~D\in \mathbb{N}$ spanned by subintervals of the coordinate axes inside the Cayley graph of $\mathbb{Z}^2$ with respect to the generating set $\left\{(1,0),(0,1)\right\}$. A \emph{square} of side length $C\in \mathbb{N}$ (or \emph{$C$-square}) is a $(C\times C)$-rectangle. To avoid confusion between $C$-squares and squares that arise as sides of cubes in a cube complex, we will subsequently refer to the latter as $2$-cubes.

\begin{theorem}[{\cite[Corollary 5]{CDEHV-08}}]\label{thm:CDEHV}
	Let $\Gamma$ be a median graph with unit length edges and let $C\in \mathbb{N}$. Then $\Gamma$ is $C$-hyperbolic with respect to the $4$-point condition for hyperbolicity if and only if it contains an isometrically embedded $C$-square, but no isometrically embedded $(C+1)$-square.
\end{theorem}

Note that the precise statement of~\cite[Corollary 5]{CDEHV-08} says that $\Gamma$ is $C$-hyperbolic if it contains no $C$-square. However, its proof actually shows that the hyperbolicity constant for $\Gamma$ is precisely the maximal side length of a square that isometrically embeds in $\Gamma$. 

In Section~\ref{sec:max-squares}, we will apply \Cref{thm:CDEHV} to the cube complex $\tX$ from \Cref{sec:Bradys-group}. This will yield the following result, which has \Cref{thm:Main-hyperbolicity-constant} as a direct consequence.

\begin{theorem}\label{thm:hyperbolicity-constant}
    The $1$-skeleton $\tX^{(1)}$ of the cube complex from \Cref{sec:Bradys-group} contains an isometrically embedded $4$-square, but no isometrically embedded $5$-square. In particular, the best possible hyperbolicity constant for the $4$-point condition is $\delta= 4$.
\end{theorem}

As a consequence we obtain an upper bound on the hyperbolicity constant for the thin triangle condition, completing the proof of \Cref{thm:Dehn-Bradys-group}.
\begin{corollary}\label{cor:hyperbolicity-constant}
    The $1$-skeleton $\tX^{(1)}$ of the cube complex $\tX$ from \Cref{sec:Bradys-group} has $16$-thin triangles.
\end{corollary}
\begin{proof}
    This is a direct consequence of Theorem~\ref{thm:hyperbolicity-constant} and Lemma~\ref{lem:thin-hyperbolicity}. 
\end{proof}

\begin{remark}\label{rmk:improving-delta}
    Since we deduce the upper bound on the constant for the thin triangle condition for hyperbolicity from the very general Lemma~\ref{lem:thin-hyperbolicity}, it is possible that in our situation this constant $\delta$ can be improved. This would allow us to improve the upper bound in \Cref{thm:Main-Dehn}. We observe that the boundary of the $4$-square in Figure~\ref{fig:4-square} defines a (degenerate) geodesic triangle $\Delta$ which is not $\delta$-thin for any $\delta<8$. Thus, the best possible upper bound on $\delta$ we can hope to obtain is $8$.
\end{remark}

\subsection{Maximal embedded squares}\label{sec:max-squares}
In this section we will prove Theorem~\ref{thm:hyperbolicity-constant}, showing that we obtain a hyperbolicity constant of $4$ in the $4$-point definition of hyperbolicity for the cube complex from \Cref{sec:Bradys-group}. One can interpret our arguments as a combinatorial version of those employed in~\cite{Bra-99} to show that there are no isometrically embedded Euclidean planes in $\tX$. 

The main step in the proof of \Cref{thm:hyperbolicity-constant} is an analysis of isometrically embedded rectangles in the $1$-skeleton of $\tX$, which will allow us to determine the maximal size of an isometrically embedded square. To achieve this, we will start from some unramified vertex $v\in \tX$ and consider all possible developments of rectangles starting from this vertex. Since $\tX$ is symmetric with respect to permutations of the three coordinates of $\Theta^3$, the possible extensions we need to consider will reduce to relatively few cases.

Before starting our analysis, we record some combinatorial constraints on the ramification loci of cubes in $\tX$, which will play an important role. Every $2$-cube of $\tX$ has precisely one ramified edge and precisely one ramified vertex, while every $3$-cube of $\tX$ has precisely three pairwise non-intersecting ramified edges (see Figure~\ref{fig:squares-and-cubes}). Moreover, every maximal cubical Euclidean neighbourhood of an unramified vertex looks as the right-most cube with side lengths $2$ in Figure~\ref{fig:squares-and-cubes}. 

When applying these combinatorial constraints, we will often focus on \emph{interior (ramified)} vertices, edges or subgraphs, which are (ramified) vertices, edges or subgraphs that have empty intersection with the boundary of the rectangle. We will illustrate our proofs with several figures, in which we will draw ramified components in red and all other components in black. Moreover, we will shade embedded rectangles in yellow, if the rectangle under consideration is not uniquely determined by the figure.

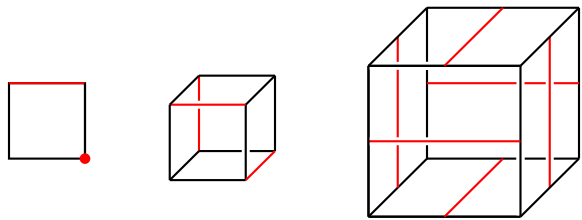
\begin{figure}
\centering
\begin{tikzpicture}
\draw[thick] (0,0) -- (1,0) -- (1,1) -- (0,1) -- cycle;  
\draw[red, thick] (0,1) -- (1,1);  
\draw[thick] (1,0) -- (1,1);  
\fill[red] (1,0) circle (2pt);  

\begin{scope}[xshift=2.5cm, yshift=0.1cm]
    \coordinate (A) at (0,0,0);
    \coordinate (B) at (1,0,0);
    \coordinate (C) at (1,1,0);
    \coordinate (D) at (0,1,0);
    \coordinate (E) at (0,0,1);
    \coordinate (F) at (1,0,1);
    \coordinate (G) at (1,1,1);
    \coordinate (H) at (0,1,1);

    \draw[red, thick] (D) -- (A);
    \draw[thick] (A) -- (B);
    \draw[thick] (B) -- (C);
    \draw[thick] (C) -- (D);

    \draw[white, line width = 3pt] (G) -- (F);
    \draw[white, line width = 3pt] (G) -- (H);
    \draw[thick] (E) -- (F);
    \draw[thick] (F) -- (G);
    \draw[red, thick] (G) -- (H);
    \draw[thick] (H) -- (E);
    \draw[thick] (A) -- (E);
    \draw[red, thick] (B) -- (F);
    \draw[thick] (C) -- (G);
    \draw[thick] (D) -- (H);
\end{scope}

\begin{scope}[xshift=5.5cm]
    \coordinate (A) at (0,0,0);
    \coordinate (B) at (2,0,0);
    \coordinate (C) at (2,2,0);
    \coordinate (D) at (0,2,0);
    \coordinate (E) at (0,0,2);
    \coordinate (F) at (2,0,2);
    \coordinate (G) at (2,2,2);
    \coordinate (H) at (0,2,2);
    
    \coordinate (a) at (1,0,0);
    \coordinate (b) at (1,2,0);
    \coordinate (c) at (1,0,2);
    \coordinate (d) at (1,2,2);
    \coordinate (e) at (0,1,0);
    \coordinate (f) at (2,1,0);
    \coordinate (g) at (0,1,2);
    \coordinate (h) at (2,1,2);
    \coordinate (i) at (0,0,1);
    \coordinate (j) at (2,0,1);
    \coordinate (k) at (0,2,1);
    \coordinate (l) at (2,2,1);

    \draw[thick] (D) -- (A);
    \draw[thick] (A) -- (B);
    \draw[thick] (B) -- (C);
    \draw[thick] (C) -- (D);

    \draw[white, line width = 3pt] (G) -- (F);
    \draw[white, line width = 3pt] (G) -- (H);
    
    \draw[thick] (E) -- (F);
    \draw[thick] (H) -- (E);
    \draw[thick] (A) -- (E);

    \draw[red, thick] (i) -- (k); 
    \draw[red, thick] (e) -- (f); 

    \draw[white, line width = 3pt] (G) -- (F);
    \draw[white, line width = 3pt] (G) -- (H);
    \draw[white, line width = 3pt] (l) -- (j);
    \draw[white, line width = 3pt] (g) -- (h);
    \draw[thick] (F) -- (G);
    \draw[thick] (G) -- (H);

    \draw[red, thick] (g) -- (h); 
    \draw[red, thick] (l) -- (j); 

    \draw[thick] (G) -- (H);
    \draw[thick] (B) -- (F);
    \draw[thick] (C) -- (G);
    \draw[thick] (D) -- (H);
    \draw[thick] (E) -- (F);
    \draw[thick] (E) -- (H);

        \draw[red, thick] (a) -- (c);
    \draw[red, thick] (b) -- (d);

\end{scope}

\end{tikzpicture}    \caption{Ramification loci in Euclidean $2$-cubes and $3$-cubes of $\tX$}
    \label{fig:squares-and-cubes}
\end{figure}

To prove that $\tX^{(1)}$ does not contain any isometrically embedded $5$-squares we will require some auxiliary results and observations.

We start by observing that embedded $2$-squares in $\tX^{(1)}$ are in one to one correspondence with $4$-cycles in the link of their central vertex. Thus, Lemma~\ref{lem:4-cycles} provides us with the following description of the possible $4$-cycles in links of $\tX$, where $\Gamma$ is the graph from Lemma~\ref{lem:4-cycles}.

If $v\in \tX$ is an unramified vertex, then there are two types of $4$-cycles in $\lk(v)$:

\begin{itemize}[leftmargin=2.5cm]
    \item[(Type U.1)] a $4$-cycle without a diagonal edge;
    \item[(Type U.2)] a $4$-cycle with a diagonal edge.
\end{itemize}

Here Type U.1 corresponds to a square that extends to a Euclidean square in $\tX$, while Type U.2 corresponds to a square that forms two sides of a rectangular parallelepiped in $\tX$. We illustrate both cases in Figure~\ref{fig:unramified-2-squares}.

\begin{figure}
    \centering
\begin{tikzpicture}
	\begin{scope}
\coordinate[label=above:] (A) at (0,0,0);
\coordinate[label=above:] (B) at (2,0,0);
\coordinate[label=above:] (C) at (2,2,0);
\coordinate[label=above:] (D) at (0,2,0);
\coordinate[label=above:] (E) at (0,0,4);
\coordinate[label=above:] (F) at (2,0,4);
\coordinate[label=above:] (G) at (2,2,4);
\coordinate[label=above:] (H) at (0,2,4);
\coordinate[label=above:] (I) at (0,0,2);
\coordinate[label=above:] (J) at (2,0,2);
\coordinate[label=above left:{$v$}] (K) at (2,2,2);
\coordinate[label=above:] (L) at (0,2,2);
\coordinate[label=above:] (M) at (4,0,0);
\coordinate[label=above:] (N) at (4,2,0);
\coordinate[label=above:] (O) at (4,0,2);
\coordinate[label=above:] (P) at (4,2,2);
\coordinate[label=above:] (Q) at (4,0,4);
\coordinate[label=above:] (R) at (4,2,4);

\draw[thick] (B) -- (C);

\draw[thick] (A) -- (M);
\draw[thick] (D) -- (A);

\draw[white, line width=3pt] (L) -- (P);
\draw[red, very thick] (L) -- (I);
\draw[white, line width=3pt] (O) -- (P);
\draw[white, line width=3pt] (J) -- (K);
\draw[thick] (J) -- (K);

\draw[white, line width=3pt] (H) -- (R);

\draw[thick] (I) -- (O);

\draw[white, line width=3pt] (F) -- (G);

\draw[thick] (M) -- (N);
\draw[thick] (E) -- (Q);
\draw[thick] (Q) -- (R);
\draw[red, very thick] (R) -- (H);
\draw[thick] (H) -- (E);
\draw[red, very thick] (O) -- (P);
\draw[thick] (P) -- (L);

\draw[thick] (A) -- (E);
\draw[red, very thick] (B) -- (F);
\draw[thick] (C) -- (G);
\draw[thick] (D) -- (H);
\draw[thick] (R) -- (N);
\draw[thick] (Q) -- (M);

\draw[thick] (F) -- (G);

\draw[red, very thick] (N) -- (D);
\fill[color=black] (2, 2, 2) circle [radius=1.5pt];

\fill[opacity=0.5, yellow] (H) -- (D) -- (N) -- (R) -- cycle;
\node[anchor=north] at (2, 0, 4) {{Type U.1}};	

	\end{scope}
	
	\begin{scope}[xshift=6.5cm]
\coordinate[label=above:] (A) at (0,0,0);
\coordinate[label=above:] (B) at (2,0,0);
\coordinate[label=above:] (C) at (2,2,0);
\coordinate[label=above:] (D) at (0,2,0);
\coordinate[label=above:] (E) at (0,0,4);
\coordinate[label=above:] (F) at (2,0,4);
\coordinate[label=above:] (G) at (2,2,4);
\coordinate[label=above:] (H) at (0,2,4);
\coordinate[label=above:] (I) at (0,0,2);
\coordinate[label=above:] (J) at (2,0,2);
\coordinate[label=above left:{$v$}] (K) at (2,2,2);
\coordinate[label=above:] (L) at (0,2,2);

\draw[thick] (A) -- (B);
\draw[thick] (B) -- (C);
\draw[red, very thick] (C) -- (D);
\draw[thick] (D) -- (A);
\draw[thick] (L) -- (I);
\draw[thick] (I) -- (J);
\draw[white, line width=3pt] (J) -- (K);
\draw[thick] (J) -- (K);
\draw[thick] (I) -- (J);
\draw[white, line width=3pt] (L) -- (K);
\draw[thick] (K) -- (L);
\draw[red, very thick] (L) -- (I);

\draw[white, line width=3pt] (F) -- (G);

\draw[white, line width=3pt] (G) -- (H);
\draw[red, very thick] (G) -- (H);
\draw[thick] (F) -- (G);
\draw[thick] (E) -- (F);
\draw[thick] (A) -- (E);
\draw[red, very thick] (B) -- (F);
\draw[thick] (D) -- (H);
\draw[thick] (C) -- (G);
\fill[color=black] (2, 2, 2) circle [radius=1.5pt];

\fill[opacity=0.5, yellow] (C) -- (B) -- (F) -- (G) -- cycle;
\fill[opacity=0.5, yellow] (H) -- (D) -- (C) -- (G) -- cycle;
\draw[thick] (H) -- (E);

\node[anchor=north] at (1, 0, 4) {{Type U.2}};	

	\end{scope}

\end{tikzpicture}
    
    \caption{The two possibilities for $2$-squares with unramified interior vertex
    }
    \label{fig:unramified-2-squares}
\end{figure}
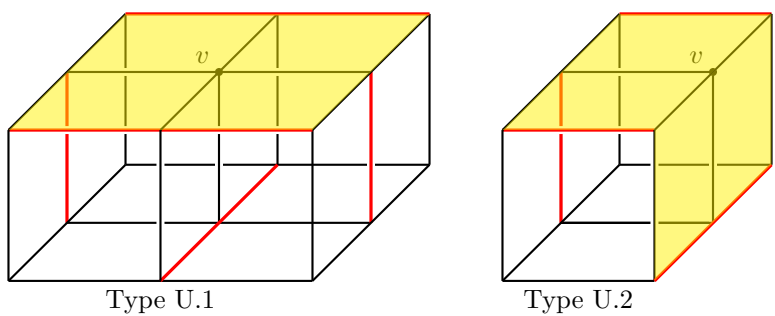

If $v\in \tX$ is a ramified vertex, then there are again two types of $4$-cycles in $\lk(v)$:

\begin{itemize}[leftmargin=2.5cm]
    \item[(Type R.1)] a $4$-cycle without a diagonal edge; and
    \item[(Type R.2)] a $4$-cycle with a diagonal edge.
\end{itemize}

Here Type R.1 corresponds to a $4$-cycle with two vertices in $\Gamma$ and two vertices in $F$ and Type R.2 corresponds to a $4$-cycle with three vertices in $\Gamma$ and one vertex in $F$. 

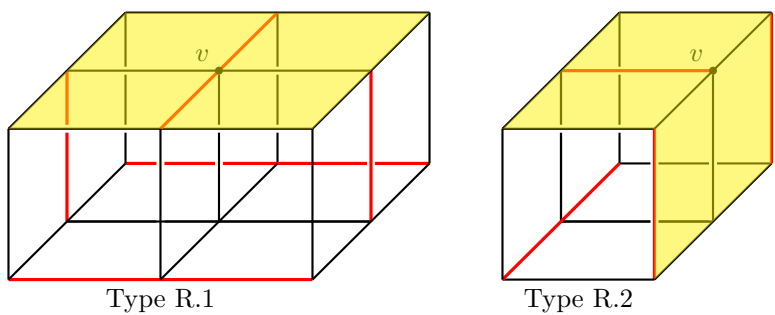
\begin{figure}
    \centering
\begin{tikzpicture}
	\begin{scope}
\coordinate[label=above:] (A) at (0,0,0);
\coordinate[label=above:] (B) at (2,0,0);
\coordinate[label=above:] (C) at (2,2,0);
\coordinate[label=above:] (D) at (0,2,0);
\coordinate[label=above:] (E) at (0,0,4);
\coordinate[label=above:] (F) at (2,0,4);
\coordinate[label=above:] (G) at (2,2,4);
\coordinate[label=above:] (H) at (0,2,4);
\coordinate[label=above:] (I) at (0,0,2);
\coordinate[label=above:] (J) at (2,0,2);
\coordinate[label=above left:{$v$}] (K) at (2,2,2);
\coordinate[label=above:] (L) at (0,2,2);
\coordinate[label=above:] (M) at (4,0,0);
\coordinate[label=above:] (N) at (4,2,0);
\coordinate[label=above:] (O) at (4,0,2);
\coordinate[label=above:] (P) at (4,2,2);
\coordinate[label=above:] (Q) at (4,0,4);
\coordinate[label=above:] (R) at (4,2,4);

\draw[thick] (B) -- (C);

\draw[red, very thick] (A) -- (M);
\draw[thick] (D) -- (A);

\draw[white, line width=3pt] (J) -- (K);
\draw[thick] (J) -- (K);

\draw[white, line width=3pt] (L) -- (P);
\draw[red, very thick] (L) -- (I);
\draw[white, line width=3pt] (O) -- (P);
\draw[white, line width=3pt] (H) -- (R);

\draw[thick] (I) -- (O);

\draw[white, line width=3pt] (F) -- (G);

\draw[thick] (M) -- (N);
\draw[red, very thick] (E) -- (Q);
\draw[thick] (Q) -- (R);
\draw[thick] (R) -- (H);
\draw[thick] (H) -- (E);
\draw[thick] (I) -- (O);
\draw[red, very thick] (O) -- (P);
\draw[thick] (P) -- (L);

\draw[thick] (A) -- (E);
\draw[thick] (B) -- (F);
\draw[red, very thick] (C) -- (G);
\draw[thick] (D) -- (H);
\draw[thick] (R) -- (N);
\draw[thick] (Q) -- (M);

\draw[thick] (F) -- (G);

\draw[thick] (N) -- (D);
\fill[color=black] (2, 2, 2) circle [radius=1.5pt];

\fill[opacity=0.5, yellow] (H) -- (D) -- (N) -- (R) -- cycle;
\node[anchor=north] at (2, 0, 4) {{Type R.1}};	

	\end{scope}
	
	\begin{scope}[xshift=6.5cm]
\coordinate[label=above:] (A) at (0,0,0);
\coordinate[label=above:] (B) at (2,0,0);
\coordinate[label=above:] (C) at (2,2,0);
\coordinate[label=above:] (D) at (0,2,0);
\coordinate[label=above:] (E) at (0,0,4);
\coordinate[label=above:] (F) at (2,0,4);
\coordinate[label=above:] (G) at (2,2,4);
\coordinate[label=above:] (H) at (0,2,4);
\coordinate[label=above:] (I) at (0,0,2);
\coordinate[label=above:] (J) at (2,0,2);
\coordinate[label=above left:{$v$}] (K) at (2,2,2);
\coordinate[label=above:] (L) at (0,2,2);

\draw[thick] (A) -- (B);
\draw[red, very thick] (B) -- (C);
\draw[thick] (C) -- (D);
\draw[thick] (D) -- (A);
\draw[thick] (L) -- (I);
\draw[thick] (I) -- (J);
\draw[white, line width=3pt] (J) -- (K);
\draw[thick] (J) -- (K);
\draw[thick] (I) -- (J);
\draw[white, line width=3pt] (L) -- (K);
\draw[red, very thick] (K) -- (L);
\draw[thick] (L) -- (I);

\draw[white, line width=3pt] (F) -- (G);

\draw[white, line width=3pt] (G) -- (H);
\draw[thick] (G) -- (H);
\draw[red, very thick] (F) -- (G);
\draw[thick] (E) -- (F);
\draw[red, very thick] (A) -- (E);
\draw[thick] (B) -- (F);
\draw[thick] (D) -- (H);
\draw[thick] (C) -- (G);
\fill[color=black] (2, 2, 2) circle [radius=1.5pt];

\fill[opacity=0.5, yellow] (C) -- (B) -- (F) -- (G) -- cycle;
\fill[opacity=0.5, yellow] (H) -- (D) -- (C) -- (G) -- cycle;
\draw[thick] (H) -- (E);

\node[anchor=north] at (1, 0, 4) {{Type R.2}};	

	\end{scope}

\end{tikzpicture}
    
    \caption{The two possibilities for $2$-squares with ramified interior vertex
    }
    \label{fig:ramified-2-squares}
\end{figure}

We record the following consequence.

\begin{lemma}\label{lem:no-int-ram-vertices}
    An isometrically embedded rectangle in $\tX^{(1)}$ can not contain an interior ramified vertex.
\end{lemma}
\begin{proof}
    An isolated ramified vertex $v$ would correspond to a $4$-cycle in $\lk(v)$ all of whose vertices are in $\Gamma$, but by our preceding discussion such $4$-cycles can not exist.
\end{proof}

We deduce the following generalisation of Lemma~\ref{lem:no-int-ram-vertices}, which is a key tool in our proof that there is no embedded $5$-square in $\tX^{(1)}$.

\begin{proposition}\label{prop:no-interior-ramification}
    An isometrically embedded rectangle in $\tX^{(1)}$ can not contain a non-empty interior ramified subgraph.
\end{proposition}
\begin{proof}
Due to the structure of ramified components inside $\tX$ described above, any interior ramified connected subgraph of a rectangle intersects every $2$-cube of $\tX$ in at most one edge. Thus, it has to be a straight path (i.e.\ it is isometric to a Euclidean line segment when viewed as a subspace of $\tX$).

By Lemma~\ref{lem:no-int-ram-vertices} we may assume that the path has length at least one. Thus, the $2$-squares around its initial and terminal vertex correspond to $4$-cycles of Type R.2. In particular, their link has a diagonal. Since $\tX$ is ${\rm CAT}(0)$ we deduce that both $2$-squares lie on the boundary of a rectangular $(1\times 1 \times 2)$-parallelepiped, similar to the one depicted in Figure~\ref{fig:unramified-2-squares}. On the other hand the $2$-squares around any interior vertex of the path have to be of Type R.1. Thus, one of their sides can be identified with one of the faces of the $(2\times 2 \times 2)$-cube in \Cref{fig:squares-and-cubes}. Combined with the observation that every non-ramified edge of $\tX$ lies in precisely four $3$-cubes, it then follows inductively that a neighbourhood of the straight path in the rectangle would have to be embedded in a rectangular parallelepiped in one of the two possible ways shown in Figure~\ref{fig:no-interior-paths}.

Option 1 is impossible, since then the rectangle would not be isometrically embedded in $\tX^{(1)}$, while Option 2 is impossible, because the parallelepiped would then contain an isometrically embedded $2$-square with an interior ramified vertex, contradicting Lemma~\ref{lem:no-int-ram-vertices}.

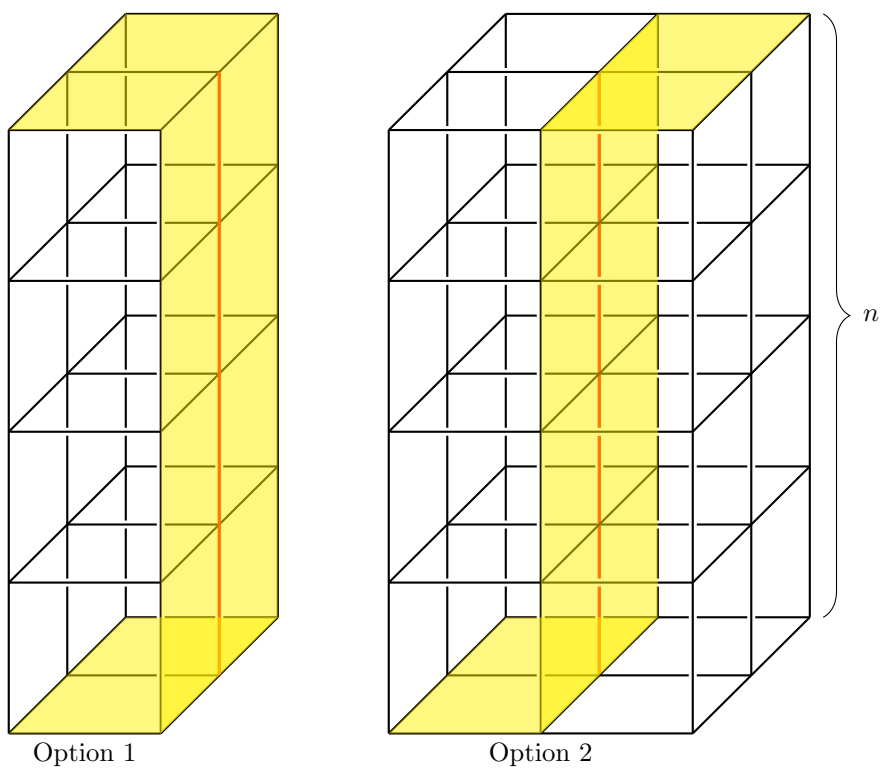
\begin{figure}
    \centering
    \begin{tikzpicture}
	\begin{scope}
			\foreach \x in {0,1} {
				\foreach \y in {0, 1, 2, 3, 4}{
					\foreach \z in {0, 1, 2}{
			\coordinate[label=above:] (\x\y\z) at (2*\x, 2*\y, 2*\z);
		}}}

		\foreach \t in {0, 1, 2, 3, 4}{
				\draw[thick] (0\t0) -- (1\t0);}
		\foreach \t in {0, 1}{
			\draw[thick] (\t00) -- (\t40);}
			
		\foreach \t in {0, 1, 2, 3, 4}{
			\draw[white, line width=3pt] (0\t1) -- (1\t1);}
		
		\foreach \t in {0, 1}{
			\draw[white, line width=3pt] (\t01) -- (\t41);}

		\foreach \t in {0, 1, 2, 3, 4}{
			\draw[thick] (0\t1) -- (1\t1);}
				
		\foreach \t in {0, 1}{
			\draw[thick] (\t01) -- (\t41);}

		\foreach \t in {0, 1}{
			\draw[white, line width=3pt] (\t02) -- (\t42);}
		\foreach \t in {0, 1, 2, 3, 4}{
			\draw[white, line width=3pt] (0\t2) -- (1\t2);}
		
		\foreach \t in {0, 1}{
			\draw[thick] (\t02) -- (\t42);}

		\foreach \t in {0, 1, 2, 3, 4}{
			\draw[thick] (0\t2) -- (1\t2);}

		\foreach \s in {0, 1, 2, 3, 4}{
			\draw[thick] (0\s0) -- (0\s2);}
		\foreach \s in {0, 1, 2, 3, 4}{
			\draw[thick] (1\s0) -- (1\s2);}
			
		\draw[red, very thick] (141) -- (101);
		\fill[opacity=0.5, yellow] (002) -- (000) -- (100) -- (102) -- cycle;
		\fill[opacity=0.5, yellow] (102) -- (100) -- (140) -- (142) -- cycle;
		\fill[opacity=0.5, yellow] (040) -- (140) -- (142) -- (042) -- cycle;
		
		\node[anchor=north] at (1, 0, 4) {{Option 1}};

	\end{scope}
	
		\begin{scope}[xshift=5cm]
		\foreach \x in {0,1, 2} {
			\foreach \y in {0, 1, 2, 3, 4}{
				\foreach \z in {0, 1, 2}{
					\coordinate[label=above:] (\x\y\z) at (2*\x, 2*\y, 2*\z);
		}}}

		\foreach \t in {0, 1, 2, 3, 4}{
			\draw[thick] (0\t0) -- (2\t0);}
		\foreach \t in {0, 1, 2}{
			\draw[thick] (\t00) -- (\t40);}
		
		\foreach \t in {0, 1, 2, 3, 4}{
			\draw[white, line width=3pt] (0\t1) -- (2\t1);}
		
		\foreach \t in {0, 1, 2}{
			\draw[white, line width=3pt] (\t01) -- (\t41);}

		\foreach \t in {0, 1, 2, 3, 4}{
			\draw[thick] (0\t1) -- (2\t1);}
		
		\foreach \t in {0, 1, 2}{
			\draw[thick] (\t01) -- (\t41);}
		
		\draw[red, very thick] (141) -- (101);
		\foreach \t in {0, 1, 2}{
			\draw[white, line width=3pt] (\t02) -- (\t42);}
		\foreach \t in {0, 1, 2, 3, 4}{
			\draw[white, line width=3pt] (0\t2) -- (2\t2);}
		
		\foreach \t in {0, 1, 2}{
			\draw[thick] (\t02) -- (\t42);}
						
		\foreach \t in {0, 1, 2, 3, 4}{
			\draw[thick] (0\t2) -- (2\t2);}

		\foreach \s in {0, 1, 2, 3, 4}{
			\draw[thick] (0\s0) -- (0\s2);}
		\foreach \s in {0, 1, 2, 3, 4}{
			\draw[thick] (1\s0) -- (1\s2);}
		\foreach \s in {0, 1, 2, 3, 4}{
			\draw[thick] (2\s0) -- (2\s2);}

		\fill[opacity=0.5, yellow] (002) -- (000) -- (100) -- (102) -- cycle;
		\fill[opacity=0.5, yellow] (102) -- (100) -- (140) -- (142) -- cycle;
		\fill[opacity=0.5, yellow] (240) -- (140) -- (142) -- (242) -- cycle;
		
		\draw[decorate,decoration={brace,amplitude=10pt,raise=5pt}] (240) -- (200) node[midway,xshift=23pt]{{$n$}};
		
		\node[anchor=north] at (2, 0, 4) {{Option 2}};

	\end{scope}
\end{tikzpicture}
    \caption{The two options for an interior ramified path of length $n$.}
    \label{fig:no-interior-paths}
\end{figure}
\end{proof}

\begin{proposition}\label{prop:no-5-squares}
    The $1$-skeleton $\tX^{(1)}$ of the cube complex from \Cref{sec:Bradys-group} does not contain any isometrically embedded $5$-square.
\end{proposition}

\begin{proof}
We will prove the assertion by showing that every possible way of extending a $2$-square in $\tX^{(1)}$ to an isometrically embedded $5$-square leads to a contradiction with either the combinatorics of $2$-cubes or \Cref{prop:no-interior-ramification}. 

We start by observing that due to the combinatorics of $2$-cubes every rectangle of side lengths $\geq 3$ will contain an unramified vertex in its interior. Thus, there is no loss of generality in starting from an unramified vertex $v$ and developing the rectangle from there. 

We first assume that the rectangle has an unramified vertex $v$ of Type U.1 in its interior. Due to the symmetries of $\tX$ it will always look as the one pictured in Figure~\ref{fig:2-square-unramified}.
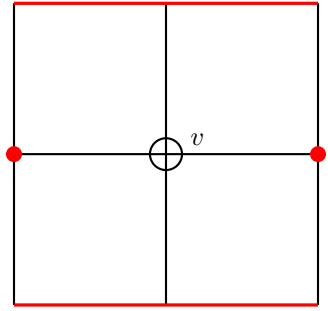
\begin{figure}
    \centering
    \begin{tikzpicture}
\draw[thick] (0,0) -- (4,0);
\draw[thick] (0,4) -- (4,4); 
\draw[thick] (0,0) -- (0,4); 
\draw[thick] (4,0) -- (4,4); 

\draw[thick] (2,0) -- (2,4); 
\draw[thick] (0,2) -- (4,2); 

\fill[red] (0,2) circle (3pt);
\fill[red] (4,2) circle (3pt); 

\draw[thick] (2,2) circle (6pt); 

\node[anchor=west] at (2.2,2.2) {$v$};

\draw[red, very thick] (0,0) -- (4,0);
\draw[red, very thick] (0,4) -- (4,4);

\end{tikzpicture}
    \caption{$2$-square with unramified interior vertex of Type U.1}
    
    \label{fig:2-square-unramified}
\end{figure}

Case by case we consider its possible extensions to the left, right, top, and bottom and show that they will never allow us to produce a $5$-square. We start by observing that the only possible extensions to the left and right correspond to a $4$-cycle of type R.2, followed by $4$-cycles of type R.1. Indeed, if the first extension was not of type R.2, then we would produce an interior ramified vertex, which can not exist by Proposition~\ref{prop:no-interior-ramification}. Any further extension to the left or right after this needs to correspond to $4$-cycles of type R.1, as a further extension by a $4$-cycle of type R.2 would produce an interior ramified graph. We conclude that all possible extensions to the left and right are subsets of the one illustrated in Figure~\ref{fig:U1-height-2-rectangle}.

\begin{figure}
    \centering
    \begin{tikzpicture}
		
		\foreach \x in {-1, 0,1,...,7, 8} {
			\draw[thick] (\x,0) -- (\x,2); 
		}
		\foreach \y in {0,1, 2} {
			\draw[thick] (-1,\y) -- (8,\y);
		}
		\foreach \y in {0 ,1, 2} {
			\draw[dashed, thick] (-2,\y) -- (-1,\y);
			\draw[dashed, thick] (8,\y) -- (9,\y);
		}

		\draw[red, very thick] (3,0) -- (5,0); 
		\draw[red, very thick] (3,2) -- (5,2); 
		\draw[red, very thick] (-1,1) -- (3,1);
		\draw[red, very thick] (5,1) -- (8,1); 
						
		\fill[red] (3,1) circle (3pt);
		\fill[red] (5,1) circle (3pt); 
		\fill[red] (7,2) circle (3pt); 
		\fill[red] (7,0) circle (3pt); 
		
		\fill[red] (1,2) circle (3pt); 
		\fill[red] (1,0) circle (3pt); 
		
		\fill[red] (-1,2) circle (3pt); 
		\fill[red] (-1,0) circle (3pt);
		
		\draw[thick] (4,1) circle (4pt); 
		\node[anchor=west] at (4,1.2) {$v$};

		\draw[decorate,decoration={brace,amplitude=10pt,mirror,raise=5pt}] (2,0) -- (6,0) node[midway,yshift=-20pt]{{Case 1}};
		\draw[decorate,decoration={brace,amplitude=10pt,raise=5pt}] (3,2) -- (8,2) node[midway,yshift=20pt]{{Case 2}};
		
	\end{tikzpicture}
    \caption{Extension of a $2$-square with an interior Type U.1 vertex $v$ to the left and right}
    \label{fig:U1-height-2-rectangle}
\end{figure}
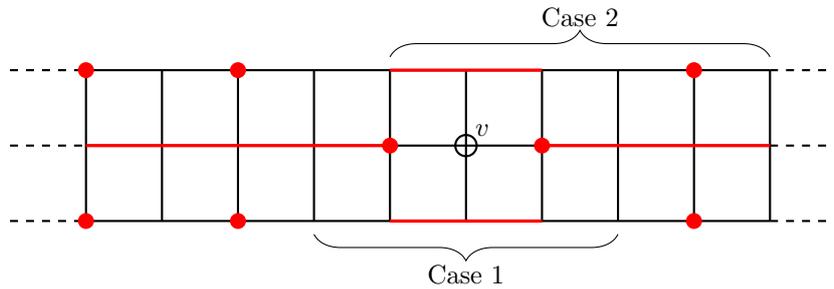

We will show that neither of the rectangles marked as Cases 1 and 2 in Figure~\ref{fig:U1-height-2-rectangle} can be extended to contain a $5$-square. Since there are no other possibilities to produce $5$-squares in a rectangle with an unramified interior vertex of Type U.1, this will reduce our considerations to the case when all unramified interior vertices are of type U.2.

In Case 1 of Figure~\ref{fig:U1-height-2-rectangle} the combinatorics of $2$-cubes immediately imply that any extension by an additional row to the top or bottom will produce an interior ramified path of length two, contradicting Proposition~\ref{prop:no-interior-ramification}.

In Case 2 of Figure~\ref{fig:U1-height-2-rectangle} it suffices to show that we can not extend by two rows to the top or bottom, as this would be the only way to obtain a $5$-square inside the rectangle. Figure~\ref{fig:3-by-5-rectangle} shows the only possible extension by one row to the bottom (an extension to the top would be symmetric). To see this we first observe that by Proposition~\ref{prop:no-interior-ramification} the vertex $w$ has to be of Type R.2. The combinatorics of ramified sets inside $2$-cubes then only allow this unique extension to a $(3\times 5)$-rectangle. A further extension by one row to the bottom would produce an interior ramified edge and thus contradict Proposition~\ref{prop:no-interior-ramification}.

\begin{figure}
    \centering
    \begin{tikzpicture}

\draw[thick] (0,0) grid (5,3);

\draw[red, very thick] (0,3) -- (2,3);    
\draw[red, very thick] (2,0) -- (3,0);  
\draw[red, very thick] (2,2) -- (5,2);   
\draw[red, very thick] (0,1) -- (2,1); 
\draw[red, very thick] (4,0) -- (4,1); 

\fill[red] (0,2) circle (3pt);  
\draw[black] (1,2) circle (3pt);
\fill[red] (2,2) circle (3pt);  
\fill[red] (4,3) circle (3pt); 
\fill[red] (2,0) circle (3pt); 
\fill[red] (3,0) circle (3pt); 
\fill[red] (4,1) circle (3pt); 
\fill[red] (5,0) circle (3pt); 

\node at (1.2,2.2) {$v$}; 
\node at (4.2,0.8) {$w$}; 

\end{tikzpicture}
    \caption{The only possibility for a $(3\times5)$-rectangle with an interior vertex $v$ of Type U.1.}
    \label{fig:3-by-5-rectangle}
\end{figure}
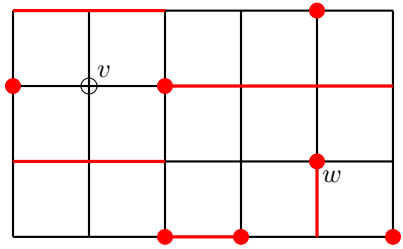

We have reduced to the case when all interior unramified vertices are of Type U.2 and thus look as in Figure~\ref{fig:U2-2-square}. The only possible extension by one column to the left corresponds to a vertex of Type R.2 and thus looks as pictured in Figure~\ref{fig:U2-left}. Extending the rectangle further to the top or bottom would create an interior ramified edge, which is impossible.

\begin{figure}
    \centering
    \begin{tikzpicture}

	\foreach \x in {0,1,2} {
		\draw[thick] (\x,0) -- (\x,2); 
	}
	\foreach \y in {0,1,2} {
		\draw[thick] (0,\y) -- (2,\y);
	}
	
	\draw[red, very thick] (2,0) -- (2,2);
	\draw[red, very thick] (0,0) -- (1,0);
	\draw[red, very thick] (0,2) -- (1,2);

	\fill[red] (0,1) circle (3pt);
	\fill[red] (1,0) circle (3pt);
	\fill[red] (1,2) circle (3pt);	

	\draw[thick] (1,1) circle (4pt); 
	\node[anchor=south west] at (1,1) {$v$};

\end{tikzpicture}
    \caption{$2$-square with unramified interior vertex of Type U.2.}
    \label{fig:U2-2-square}
\end{figure}
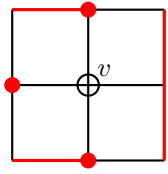

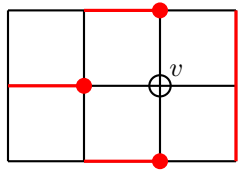
\begin{figure}
    \centering
    \begin{tikzpicture}

\foreach \x in {-1, 0,1,2} {
	\draw[thick] (\x,0) -- (\x,2); 
}
\foreach \y in {0,1,2} {
	\draw[thick] (-1,\y) -- (2,\y);
}

\draw[red, very thick] (2,0) -- (2,2);
\draw[red, very thick] (0,0) -- (1,0);
\draw[red, very thick] (0,2) -- (1,2);
\draw[red, very thick] (-1,1) -- (0,1);

\fill[red] (0,1) circle (3pt);
\fill[red] (1,0) circle (3pt);
\fill[red] (1,2) circle (3pt);	

\draw[thick] (1,1) circle (4pt); 
\node[anchor=south west] at (1,1) {$v$};
	
\end{tikzpicture}
    \caption{The only possible extension of Figure~\ref{fig:U2-2-square} to the left.}
    \label{fig:U2-left}
\end{figure}

Thus, the only remaining possibility is an extension of the $2$-square in Figure~\ref{fig:U2-2-square} to the right. There are two possibilities for doing so, which are pictured in Figure~\ref{fig:U2-right}.

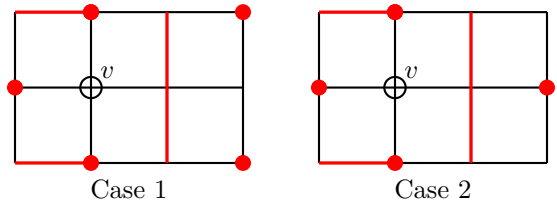
\begin{figure}
    \centering
    \begin{tikzpicture}
	\begin{scope}
	
		\foreach \x in {0,1,2, 3} {
		\draw[thick] (\x,0) -- (\x,2); 
	}
	\foreach \y in {0,1,2} {
		\draw[thick] (0,\y) -- (3,\y);
	}
	
	\draw[red, very thick] (2,0) -- (2,2);
	\draw[red, very thick] (0,0) -- (1,0);
	\draw[red, very thick] (0,2) -- (1,2);

	\fill[red] (0,1) circle (3pt);
	\fill[red] (1,0) circle (3pt);
	\fill[red] (1,2) circle (3pt);	
	\fill[red] (1,2) circle (3pt);
	\fill[red] (3,2) circle (3pt);
	\fill[red] (3,0) circle (3pt);
	
	\draw[thick] (1,1) circle (4pt); 
	\node[anchor=south west] at (1,1) {$v$};
		\node[anchor=north] at (1.5,-.1) {{Case 1}};
	
	\end{scope}
	
	\begin{scope}[xshift =4cm]
		
		\foreach \x in {0,1,2, 3} {
			\draw[thick] (\x,0) -- (\x,2); 
		}
		\foreach \y in {0,1,2} {
			\draw[thick] (0,\y) -- (3,\y);
		}
		
		\draw[red, very thick] (2,0) -- (2,2);
		\draw[red, very thick] (0,0) -- (1,0);
		\draw[red, very thick] (0,2) -- (1,2);

		\fill[red] (0,1) circle (3pt);
		\fill[red] (1,0) circle (3pt);
		\fill[red] (1,2) circle (3pt);	
		\fill[red] (1,2) circle (3pt);
		\fill[red] (3,1) circle (3pt);

		\draw[thick] (1,1) circle (4pt); 
		\node[anchor=south west] at (1,1) {$v$};
			\node[anchor=north] at (1.5,-.1) {{Case 2}};
		
	\end{scope}
\end{tikzpicture}
    \caption{The two possible extensions of Figure~\ref{fig:U2-2-square} to the right.}
    \label{fig:U2-right}
\end{figure}

Since any interior unramified vertex needs to be of type U.2, the only extension of Case 1 of Figure~\ref{fig:U2-right} to the right is the one pictured in Figure~\ref{fig:U2-right-Case1} with  unramified interior vertex $w$ of type U.2. We can argue as before that any further extension to the right or left would produce a rectangle as in Figure~\ref{fig:U2-left}, which can not be extended any further to the top or bottom. This rules out Case 1 of \Cref{fig:U2-right}.

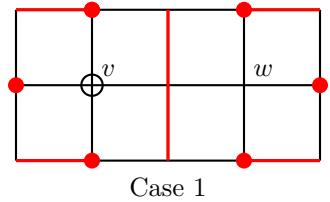
\begin{figure}
    \centering
    \begin{tikzpicture}

		\foreach \x in {0,1,2, 3, 4} {
		\draw[thick] (\x,0) -- (\x,2); 
	}
	\foreach \y in {0,1,2} {
		\draw[thick] (0,\y) -- (4,\y);
	}
	
	\draw[red, very thick] (2,0) -- (2,2);
	\draw[red, very thick] (0,0) -- (1,0);
	\draw[red, very thick] (0,2) -- (1,2);
	
	\draw[red, very thick] (3,0) -- (4,0);
	\draw[red, very thick] (3,2) -- (4,2);

	\fill[red] (0,1) circle (3pt);
	\fill[red] (1,0) circle (3pt);
	\fill[red] (1,2) circle (3pt);	
	\fill[red] (1,2) circle (3pt);
	\fill[red] (3,2) circle (3pt);
	\fill[red] (3,0) circle (3pt);
	\fill[red] (4,1) circle (3pt);
	
	\draw[thick] (1,1) circle (4pt); 
	\node[anchor=south west] at (1,1) {$v$};
	\node[anchor=south west] at (3,1) {$w$};
			\node[anchor=north] at (2,-.1) {{Case 1}};	
\end{tikzpicture}
    \caption{Extending Case 1 of Figure~\ref{fig:U2-right} further right.}
    \label{fig:U2-right-Case1}
\end{figure}

\begin{figure}
    \centering
   
\begin{tikzpicture}

	\foreach \x in {0,1,2, 3,4, 5} {
		\draw[thick] (\x,0) -- (\x,3); 
	}
	\foreach \y in {0,1,2, 3} {
		\draw[thick] (0,\y) -- (5,\y);
	}

	\draw[red, very thick] (0,1) -- (1,1);
	\draw[red, very thick] (0,3) -- (1,3);
	\draw[red, very thick] (2,0) -- (2,3);
	\draw[red, very thick] (4,0) -- (4,1);
	\draw[red, very thick] (3,2) -- (5,2);

	\fill[red] (1,3) circle (3pt);
	\fill[red] (4,3) circle (3pt);
	\fill[red] (0,2) circle (3pt);
	\fill[red] (3,2) circle (3pt);
	\fill[red] (4,1) circle (3pt);
	\fill[red] (1,1) circle (3pt);
	\fill[red] (3,0) circle (3pt);
	\fill[red] (0,0) circle (3pt);	

	\draw[thick] (1,2) circle (4pt); 
	\node[anchor=south west] at (1,2) {$v$};
	
	\node[anchor=north] at (2.5,-.1) {{Case 2}};
	
\end{tikzpicture}
    \caption{The only possible extension of Case 2 of Figure~\ref{fig:U2-right} to the right and bottom.}
    \label{fig:U2-right-Case2}
\end{figure}
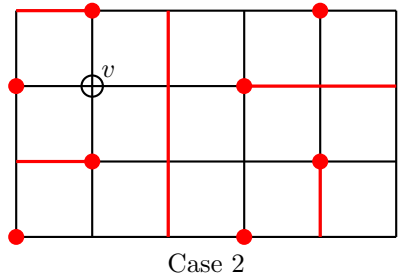

Finally, in Case 2 of Figure~\ref{fig:U2-right} we argue as before that the only possible extension to the right is given by a vertex of type R.2 followed by a vertex of type R.1. The combinatorics of ramified sets in $2$-cubes then imply that the only admissible downwards extension to a $(3\times 5)$-rectangle would have to look as in Figure~\ref{fig:U2-right-Case2}. Moreover, any further downwards extension would lead to an interior ramified subgraph, contradicting \Cref{prop:no-interior-ramification}. By symmetries the same argument shows that we can extend upwards by at most one row. We deduce that all extensions of Case 2 to embedded rectangles can at most contain an embedded $4$-square. This completes the proof.
\end{proof}

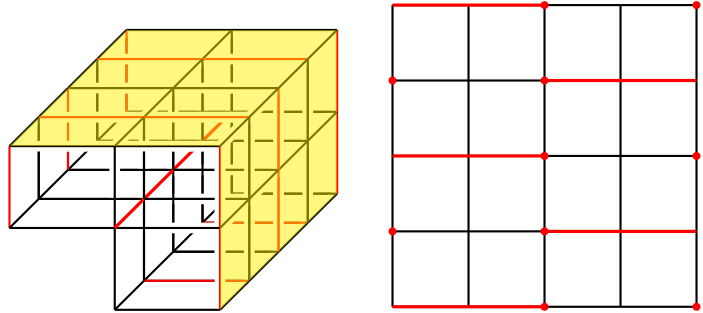
\begin{figure}
    \centering

   	\begin{tikzpicture}[scale=0.5]

				\foreach \x in {0,1, 2} {
					\foreach \y in {0, 1, 2}{
						\foreach \z in {0, 1, 2, 3, 4}{
							\coordinate[label=above:] (\x\y\z) at (2.77*\x, 2.17*\y, 2*\z);
				}}}
		\foreach \s in {0}{
			\draw[thick] (02\s) -- (22\s);
			\draw[thick] (01\s) -- (21\s);
			\draw[thick] (10\s) -- (20\s);
			\draw[thick] (02\s) -- (01\s);
			\draw[thick] (12\s) -- (10\s);
			\draw[thick] (22\s) -- (20\s);
		}
		\def\li{4pt}
		\foreach \s in {1, 2, 3, 4}{
		
			\pgfmathtruncatemacro\t{\s-1};
			\draw[white, line width=\li] (02\s) -- (22\s);
			\draw[white, line width=\li] (01\s) -- (21\s);
			\draw[white, line width=\li] (10\s) -- (20\s);
			\draw[white, line width=\li] (02\s) -- (01\s);
			\draw[white, line width=\li] (12\s) -- (10\s);
			\draw[white, line width=\li] (22\s) -- (20\s);
			\draw[white, line width=\li] (01\t) -- (01\s);
			\draw[white, line width=\li] (02\t) -- (02\s);
			\draw[white, line width=\li] (10\t) -- (10\s);
			\draw[white, line width=\li] (11\t) -- (11\s);
			\draw[white, line width=\li] (12\t) -- (12\s);
			\draw[white, line width=\li] (20\t) -- (20\s);
			\draw[white, line width=\li] (21\t) -- (21\s);
			\draw[white, line width=\li] (22\t) -- (22\s);
			
			\ifthenelse{\t=0\OR \t=2}{\draw[thick] (02\t) -- (22\t);}{\draw[thick, red] (02\t) -- (22\t);}
			\draw[thick] (01\t) -- (21\t);
			\ifthenelse{\t=2\OR \t=0}{\draw[thick] (10\t) -- (20\t);}{\draw[thick, red] (10\t) -- (20\t);}
			\ifthenelse{\t=1\OR \t=3}{\draw[thick] (02\t) -- (01\t);}{\draw[thick, red] (02\t) -- (01\t);}
			\draw[thick] (12\t) -- (10\t);
			\ifthenelse{\t=1\OR \t=3}{\draw[thick] (22\t) -- (20\t);}{\draw[thick, red] (22\t) -- (20\t);}

			\draw[white, line width=\li] (02\s) -- (22\s);
			\draw[white, line width=\li] (22\s) -- (20\s);
			
			\ifthenelse{\s=2\OR \s=4}{\draw[thick] (02\s) -- (22\s);}{\draw[thick, red] (02\s) -- (22\s);}
			\draw[thick] (01\s) -- (21\s);
			\draw[thick] (10\s) -- (20\s);
			\ifthenelse{\s=1\OR \s=3}{\draw[thick] (02\s) -- (01\s);}{\draw[thick, red] (02\s) -- (01\s);}
			\draw[thick] (12\s) -- (10\s);
			\ifthenelse{\s=1\OR \s=3}{\draw[thick] (22\s) -- (20\s);}{\draw[thick, red] (22\s) -- (20\s);}

			\draw[thick] (01\t) -- (01\s);
			\draw[thick] (02\t) -- (02\s);
			\draw[thick] (10\t) -- (10\s);
			\draw[very thick, red] (11\t) -- (11\s);
			\draw[thick] (12\t) -- (12\s);
			\draw[thick] (20\t) -- (20\s);
			\draw[thick] (21\t) -- (21\s);
			\draw[thick] (22\t) -- (22\s);
		}
		\fill[yellow, opacity=0.5] (024) -- (020) -- (220) -- (224) -- cycle;
		\fill[yellow, opacity=0.5] (204) -- (200) -- (220) -- (224) -- cycle;
			
			\begin{scope}[xshift=7cm, yshift=-3cm]
					\foreach \x in {0,1, 2, 3, 4} {
					\foreach \y in {0, 1, 2, 3, 4}{
							\coordinate[label=above:] (\x\y) at (2*\x, 2*\y);
				}}
			\draw[thick] (00) -- (40);
			\draw[thick] (01) -- (41);
			\draw[thick] (02) -- (42);
			\draw[thick] (03) -- (43);
			\draw[thick] (04) -- (44);
			\draw[thick] (00) -- (04);
			\draw[thick] (10) -- (14);
			\draw[thick] (20) -- (24);
			\draw[thick] (30) -- (34);
			\draw[thick] (40) -- (44);
			
			\draw[red, very thick] (04) -- (24);
			\draw[red, very thick] (23) -- (43);
			\draw[red, very thick] (02) -- (22);
			\draw[red, very thick] (21) -- (41);
			\draw[red, very thick] (00) -- (20);
			
			\fill[red] (20) circle (3pt);
			\fill[red] (40) circle (3pt);
			\fill[red] (21) circle (3pt);
			\fill[red] (01) circle (3pt);
			\fill[red] (22) circle (3pt);
			\fill[red] (42) circle (3pt);
			\fill[red] (23) circle (3pt);
			\fill[red] (03) circle (3pt);
			\fill[red] (24) circle (3pt);
			\fill[red] (44) circle (3pt);
			
			\end{scope}

			\end{tikzpicture}
    \caption{Example of a $4$-square that isometrically embeds in $\tX^{(1)}$}
    \label{fig:4-square}
\end{figure}

\begin{proof}[Proof of Theorem~\ref{thm:hyperbolicity-constant}]
    By \Cref{thm-Chepoi,thm:CDEHV} it suffices to determine the maximal $C>0$ such that $\tX^{(1)}$ contains an isometrically embedded $C$-square.
    
    Proposition \ref{prop:no-5-squares} shows that $\tX^{(1)}$ contains no isometrically embedded $5$-squares. It remains to show the existence of an isometrically embedded $4$-square. Figure~\ref{fig:4-square} provides an explicit example, showing its existence and thus completing the proof.
\end{proof}

\begin{remark}\label{rmk:Lodha-Kropholler-examples}
    Our proof did not use specific properties of the example from~\cite{Bra-99} that are not shared by its generalisations in~\cite{Kro-21} and~\cite{Lod-18}. Our methods thus readily extend to also provide precise computations of their hyperbolicity constants, as well as upper bounds on their Dehn functions. The hyperbolicity constant will always be 4. However, the expansion constant $C_3$ will vary and in fact go to infinity in a suitable chosen sequence of examples, since they can be constructed so that the areas of certain $4$-cycles in the ascending and descending links go to infinity (replace e.g. the ramified covering in Brady's construction by a higher degree ramified covering so that $4$-cycles lift to $n_k$-cycles and $n_k\to \infty$ as $k\to \infty$). Moreover, it may be possible to produce an example where the associated constant $C_3$ is smaller than for Brady's example, which would produce a non-hyperbolic finitely presented subgroup of a hyperbolic group with a better upper bound on its Dehn function than the one given in \Cref{thm:Main-Dehn}.
\end{remark}

\frenchspacing
\bibliography{References}
\bibliographystyle{alpha}

\end{document}